\documentclass{elsarticle}
\pdfoutput=1

\expandafter\let\csname ver@natbib.sty\endcsname\relax
\makeatletter
\let\c@author\relax
\makeatother

\usepackage{amsmath}
\usepackage{amssymb}
\usepackage{amsthm}
\usepackage{xspace}
\usepackage[mathscr]{eucal}
\usepackage[utf8]{inputenc}
\usepackage[T1]{fontenc}

\usepackage[backend=biber,maxbibnames=9,giveninits=true,isbn=false]{biblatex}
\AtEveryBibitem{\clearfield{number}}

\DeclareFieldFormat{titlecase}{\MakeTitleCase{#1}}
\newrobustcmd{\MakeTitleCase}[1]{%
  \ifthenelse{\ifcurrentfield{booktitle}\OR\ifcurrentfield{booksubtitle}%
    \OR\ifcurrentfield{maintitle}\OR\ifcurrentfield{mainsubtitle}%
    \OR\ifcurrentfield{journaltitle}\OR\ifcurrentfield{journalsubtitle}%
    \OR\ifcurrentfield{issuetitle}\OR\ifcurrentfield{issuesubtitle}%
    \OR\ifentrytype{book}\OR\ifentrytype{mvbook}\OR\ifentrytype{bookinbook}%
    \OR\ifentrytype{booklet}\OR\ifentrytype{suppbook}%
    \OR\ifentrytype{collection}\OR\ifentrytype{mvcollection}%
    \OR\ifentrytype{suppcollection}\OR\ifentrytype{manual}%
    \OR\ifentrytype{periodical}\OR\ifentrytype{suppperiodical}%
    \OR\ifentrytype{proceedings}\OR\ifentrytype{mvproceedings}%
    \OR\ifentrytype{reference}\OR\ifentrytype{mvreference}%
    \OR\ifentrytype{report}\OR\ifentrytype{thesis}}
    {#1}
    {\MakeSentenceCase{#1}}}

\usepackage{hyperref}

\pdfpagewidth=\paperwidth
\pdfpageheight=\paperheight
\usepackage[a4paper]{geometry}

\newtheorem{theorem}{Theorem}[section]
\newtheorem{lemma}[theorem]{Lemma}
\newtheorem{proposition}[theorem]{Proposition}
\newtheorem{corollary}[theorem]{Corollary}
\newtheorem{definition}[theorem]{Definition}

\newtheorem{remark}[theorem]{Remark}

\newenvironment{Proof}{\begin{proof}}{\end{proof}}
\newcommand{\dmatroid}{delta-matroid\xspace}
\newcommand{\dmatroids}{delta-matroids\xspace}

\newcommand{\sdif}{\mathop{\mathrm{\Delta}}}
\newcommand{\delm}{\setminus}

\newcommand{\ort}{\mathrm{Ort}}
\newcommand{\ortE}{\mathcal{E}}
\newcommand{\prodv}[2]{\pi_{#2}(#1)}

\renewcommand{\emptyset}{\varnothing}
\renewcommand{\vec}[1]{#1}

\begin{document}


\begin{frontmatter}

\title{Orienting Transversals and Transition Polynomials of Multimatroids}

\author{Robert Brijder\fnref{fn1}}
\ead{robert.brijder@uhasselt.be}
\fntext[fn1]{R.B.\ is a postdoctoral fellow of the Research Foundation -- Flanders (FWO).}
\address{Hasselt University, Belgium}

\begin{abstract}
Multimatroids generalize matroids, delta-matroids, and isotropic systems, and transition polynomials of multimatroids subsume various polynomials for these latter combinatorial structures, such as the interlace polynomial and the Tutte-Martin polynomial.

We prove evaluations of the Tutte-Martin polynomial of isotropic systems from Bouchet directly and more efficiently in the context of transition polynomials of multimatroids. Moreover, we generalize some related evaluations of the transition polynomial of 4-regular graphs from Jaeger to multimatroids. These evaluations are obtained in a uniform and matroid-theoretic way. We also translate the evaluations in terms of the interlace polynomial of graphs. Finally, we give an excluded-minor theorem for the class of binary tight 3-matroids (a subclass of multimatroids) based on the excluded-minor theorem for the class of binary delta-matroids from Bouchet.
\end{abstract}


\end{frontmatter}

\section{Introduction}
Inspired by the theory of circuit partitions in 4-regular graphs, Bouchet developed the notions of \dmatroids \cite{mp/Bouchet87} and isotropic systems \cite{Bouchet/87/ejc/isotropicsys}. 
Delta-matroids generalize matroids, and like matroids, delta-matroids have two kinds of minor operations: deletion and contraction. On the other hand, isotropic systems
have three types of minor operations (defined in a completely symmetric way). The notion of an isotropic system is however tied to binary spaces. With both notions in place, Bouchet then developed a matroid-theoretic notion with an arbitrary number $k$ of types of minor operations, called multimatroids or $k$-matroids \cite{DBLP:journals/siamdm/Bouchet97}, subsuming both delta-matroids and isotropic systems: delta-matroids are equivalent to $2$-matroids and isotropic systems form a strict subclass of $3$-matroids. In this way, matroids can be viewed as a subclass of $2$-matroids. Like isotropic systems, but unlike delta-matroids or matroids, the $k$ types of minor operations of multimatroids are defined in a completely symmetric way. In order to avoid clutter, we do not recall \dmatroids and isotropic systems in this paper, but instead formulate everything in terms of multimatroids.

We note that the theory of circuit partitions in 4-regular graphs is but one research topic that leads to multimatroids. As another example, embedded graphs (or ribbon graphs) have three kinds of minor operations resulting from the operations of partial duality and partial Petrie duality, which also naturally lead to $2$-matroids (using only partial duality) and $3$-matroids (using both partial duality and partial Petrie duality) \cite{EUJC/Traldi/transmat,CMNR/2016/embeddedDM}.

Various polynomials exist for each of the above combinatorial structures, and some polynomials are generalizations of others. For example, the global Tutte-Martin polynomial $M(S;y)$ of isotropic systems $S$ generalizes the Martin polynomial of 4-regular graphs, and the restricted Tutte-Martin polynomial $m(S,T;y)$ of isotropic systems ($T$ is a ``transversal'' that restricts $S$) generalizes the Tutte polynomial $T(M;x,y)$ of binary matroids $M$ on the diagonal (the case $x=y$), see \cite{TutteMartinOrientingVectors/Bouchet91}. Also, the transition polynomial of vf-safe delta-matroids \cite{BH/InterlacePolyDM/14} generalizes the Penrose polynomial \cite{EllisMonaghan/PenroseEmb/2011} and the Bollob\'as-Riordan polynomial \cite{BollobasRiordan:PolEmbedded} of embedded graphs, see \cite{CMNR/2016/embeddedDM}. The (multivariate) transition polynomial of multimatroids, defined in \cite{BenardBouchet/TutteMartin1997} and independently in \cite{BH/InterlacePolyDM/14}, has the above polynomials and others as specializations, see \cite{BH/InterlacePolyDM/14}. As such it provides a unifying matroid-theoretic framework to study these polynomials.


In this paper we prove several results from \cite{TutteMartinOrientingVectors/Bouchet91} concerning the global and the restricted Tutte-Martin polynomial of isotropic systems in terms of the transition polynomial of multimatroids. One of these results is that for all isotropic systems $S$, we have $m(S,T;3) = k|m(S,T;-1)|$ for some odd integer $k$. A special case of this result is the well-known fact that for all binary matroids $M$, $T(M;3,3) = k|T(M;-1,-1)|$ for some odd integer $k$. We stress that the generalization from $T(M;y,y)$ to $m(S,T;y)$ is significant because, unlike the former, the latter generalizes the Martin polynomial of 4-regular graphs and also because it has a much more profound influence on the theory of embedded graphs.

We significantly simplify the proof of the evaluation of $m(S,T;3)$ from \cite{TutteMartinOrientingVectors/Bouchet91} by showing that the evaluation of $m(S,T;3)$ can be treated in a way that is similar to that of the more easily obtained evaluation of $M(S;4)$. Moreover, while isotropic systems intrinsically deal with binary vector spaces, using instead the matroid-theoretic context of multimatroids allows us to factor out properties that hold for larger classes of multimatroids. This singles out one crucial property of binary tight $3$-matroids (this class of multimatroids exactly corresponds to isotropic systems \cite{BT/IsotropicMatI}) that makes the evaluations work (cf.\ Theorem~\ref{thm:bin_one_ort}). This crucial property concerns the notion of an orienting transversal which corresponds to the notion of orienting vector from isotropic systems. Additionally, we generalize some related evaluations of the transition polynomial of 4-regular graphs \cite{Jaeger1990Transition} to the more general setting of binary tight $3$-matroids.

The interlace polynomial of a graph (we allow loops, but not multiple edges) \cite{Arratia2004199} is equivalent to the restricted Tutte-Martin polynomial, and the global interlace polynomial of a graph \cite{Aigner200411} is equivalent to the global Tutte-Martin polynomial, see \cite{DBLP:journals/dm/Bouchet05}. By taking inspiration from \cite{Arratia2004199}, we translate in Subsection~\ref{ssec:conseq_int_pol} a number of  transition-polynomial evaluations formulated in terms of multimatroid notions to evaluations of interlace polynomials formulated in terms of graph-theoretical notions (such as Eulerian induced subgraphs).

Finally, based on the excluded-minor theorem of binary \dmatroids from Bouchet \cite{Bouchet_1991_67}, we give in Section~\ref{sec:excl_minor_bin_t3m} an excluded-minor theorem of binary tight $3$-matroids and show that the unique (up to isomorphism) excluded minor for this class of multimatroids does not satisfy the crucial property mentioned above (cf.\ Theorem~\ref{thm:bin_one_ort}).


\section{Multimatroids} \label{sec:mm}

We recall some notions and notation concerning multimatroids \cite{DBLP:journals/siamdm/Bouchet97}. We assume the reader is familiar with the basic notions concerning matroids, which can be found, e.g., in \cite{Welsh/MatroidBook,Oxley/MatroidBook-2nd}.

We take the terminology of multimatroids as developed by Bouchet \cite{DBLP:journals/siamdm/Bouchet97,DBLP:journals/combinatorics/Bouchet98,DBLP:journals/ejc/Bouchet01}.
A \emph{carrier} is a tuple $(U,\Omega)$ where $\Omega$ is a partition of a finite set $U$, called the \emph{ground set}. Every $\omega \in \Omega$ is called a \emph{skew class}. A \emph{transversal} $T$ of $\Omega$ is a subset of $U$ such that $|T \cap \omega| = 1$ for all $\omega \in \Omega$. A \emph{subtransversal} of $\Omega$ is a subset of a transversal of $\Omega$. We denote the set of transversals of $\Omega$ by $\mathcal{T}(\Omega)$, and the set of subtransversals of $\Omega$ by $\mathcal{S}(\Omega)$. The power set of a set $X$ is denoted by $2^X$.

Before we recall the notion of a multimatroid from \cite{DBLP:journals/siamdm/Bouchet97}, we define the notion of a semi-multimatroid.

\begin{definition} \label{def:semi-multimatroid}
A \emph{semi-multimatroid} $Z$ (described by its circuits) is a triple $(U,\Omega,\mathcal{C})$, where $(U,\Omega)$ is a carrier and $\mathcal{C} \subseteq \mathcal{S}(\Omega)$ such that for each $T \in \mathcal{T}(\Omega)$, $(T,\mathcal{C}\cap 2^T)$ is a matroid (described by its circuits).
\end{definition}

Terminology concerning $\Omega$ carries over to $Z$: hence we may, e.g., speak of a transversal of $Z$. We often simply write $U$ and $\Omega$ to denote the ground set and the partition of the semi-multimatroid $Z$ under consideration, respectively. Each $C \in \mathcal{C}$ in the definition of semi-multimatroid $Z$ is called a \emph{circuit} of $Z$. The family of circuits of $Z$ is denoted by $\mathcal{C}(Z)$. We say that $Z$ is a semi-multimatroid \emph{over} carrier $(U,\Omega)$.

The \emph{order} of a semi-multimatroid $Z = (U,\Omega,\mathcal{C})$ is $|\Omega|$. For any $X \subseteq U$, the \emph{restriction} of $Z$ to $X$, denoted by $Z[X]$, is the semi-multimatroid $(X,\Omega',\mathcal{C}\cap 2^{X})$ with $\Omega' = \{ \omega \cap X \mid \omega \cap X \neq \emptyset, \omega \in \Omega\}$. For a set $X$, the \emph{deletion} of $X$ from $Z$, denoted by $Z-X$, is $Z[U\setminus X]$. If $X$ is a subtransversal, then we identify $Z[X]$ with the matroid $(X,\mathcal{C}\cap 2^{X})$ since $\Omega' = \{\{u\} \mid u \in X\}$ captures no additional information. Note that $C$ is a circuit of $Z$ if and only if $C$ is a circuit of some matroid $Z[T]$ with $T \in \mathcal{T}(\Omega)$. Similarly, we say that $I \in \mathcal{S}(\Omega)$ is an \emph{independent set} of $Z$ if $I$ is an independent set of some matroid $Z[T]$. The family of independent sets of $Z$ is denoted by $\mathcal{I}(Z)$. The \emph{rank} of $S \in \mathcal{S}(\Omega)$ in $Z$, denoted by $r_Z(S)$, is the rank $r(Z[S])$ of the matroid $Z[S]$. The \emph{nullity} of $S \in \mathcal{S}(\Omega)$ in $Z$, denoted by $n_Z(S)$, is $n_Z(S) = |S| - r_Z(S)$, i.e., the nullity $n(Z[S])$ of the matroid $Z[S]$. A \emph{basis} $B$ of $Z$ is an element of $\mathcal{I}(Z)$ that is maximal with respect to inclusion. The family of bases of $Z$ is denoted by $\mathcal{B}(Z)$. Care must be taken regarding bases: a basis $B$ of a matroid $Z[T]$ is not necessarily a basis of $Z$.

For a function $f: X \to Y$, we define for $X' \subseteq X$, $f(X') = \cup_{x \in X'} \{f(x)\}$. For semi-multimatroids $Z_1 = (U_1,\Omega_1,\mathcal{C}_1)$ and $Z_2 = (U_2,\Omega_2,\mathcal{C}_2)$, we say that $\varphi: U_1 \to U_2$ is an \emph{isomorphism} from $Z_1$ to $Z_2$ if $\varphi$ is a one-to-one correspondence such that (1) for all $\omega \in \Omega_1$, we have $\varphi(\omega) \in \Omega_2$ and (2) for all $X \in \mathcal{S}(\Omega_1)$, $X \in \mathcal{C}_1$ if and only if $\varphi(X) \in \mathcal{C}_2$.

The \emph{minor} of $Z$ induced by $X \in \mathcal{S}(\Omega)$, denoted by $Z|X$, is the semi-multimatroid $(U',\Omega',\mathcal{C}')$, where $\Omega' = \{\omega \in \Omega \mid \omega \cap X = \emptyset\}$, $U' = \bigcup \Omega'$, and $\mathcal{C}' \subseteq \mathcal{S}(\Omega')$ such that $C \in \mathcal{C}'$ if and only if $C \in \mathcal{C}(Z[T \cup X] \slash X)$ for some $T \in \mathcal{T}(\Omega')$. (As usual, matroid contraction is denoted by $\slash$.) As shown in \cite{DBLP:journals/combinatorics/Bouchet98}, $Z|X$ is indeed a semi-multimatroid. For any set $U''$ disjoint from $X$, we have $(Z-U'')|X = (Z|X)-U''$. In case $X = \{u\}$ is a singleton, we also write $Z|u$ to denote $Z|\{u\}$. An element $u \in U$ is called \emph{singular} in $Z$ if $\{u\}$ is a circuit of $Z$.

It is unfortunate that the standard way (introduced by Bouchet in \cite{DBLP:journals/combinatorics/Bouchet98}) to denote a multimatroid minor (i.e., $Z|X$) clashes with the usual way to denote matroid restriction. Therefore, in order to minimize confusion with multimatroid minors, we denote in this paper the restriction of a matroid $M$ to a subset $X$ of its ground set by $M[X]$ (which is compatible with the notation of multimatroid restriction $Z[X]$).

Let us denote the ground set of a matroid $M$ by $E(M)$. Recall that for a matroid $M$ and disjoint subsets $X$ and $Y$ of $E(M)$, we have that the nullity $n_{M \slash X}(Y)$ of $Y$ in $M \slash X$ is equal to $n_M(X \cup Y) - n_M(X)$. Therefore, for all $S \in \mathcal{S}(\Omega')$, $n_{Z|X}(S) = n_Z(S \cup X) - n_Z(X)$.

A semi-multimatroid $Z$ is called \emph{nondegenerate} if $|\omega| > 1$ for all $\omega \in \Omega$.

Let $(U,\Omega)$ be a carrier. Then $p \subseteq \omega \in \Omega$ with $|p|=2$ is called a \emph{skew pair} of $\omega$. We now recall the notion of a multimatroid.
\begin{definition}
A semi-multimatroid $Z = (U,\Omega,\mathcal{C})$ is called a \emph{multimatroid} if for all $C_1,C_2 \in \mathcal{C}$, $C_1 \cup C_2$ does not include precisely one skew pair.
\end{definition}

Like matroids, semi-multimatroids and multimatroids can also be defined in terms of rank and independent sets, see \cite{DBLP:journals/siamdm/Bouchet97}.

Note that if $Z$ is a multimatroid, then each skew class contains at most one singular element. A skew class that contains a singular element is called \emph{singular}.

\begin{theorem} \label{thm:char_mm}
Let $Z = (U,\Omega,\mathcal{C})$ be a semi-multimatroid. Then the following conditions are equivalent:
\begin{enumerate}
\item \label{item:char_mm_def} $Z$ is a multimatroid,

\item \label{item:char_mm_coloop} for every $S \in \mathcal{S}(\Omega)$ and $\omega \in \Omega$ with $\omega \cap S = \emptyset$, there is at most one $x \in \omega$ with $\mathcal{C}(Z[S \cup \{x\}]) \neq \mathcal{C}(Z[S])$,

\item \label{item:char_mm_nullity} for every $S \in \mathcal{S}(\Omega)$ and $\omega \in \Omega$ with $\omega \cap S = \emptyset$, there is at most one $x \in \omega$ with $n_Z(S \cup \{x\}) \neq n_Z(S)$,

\item \label{item:char_mm_nullity_restr} for every $S \in \mathcal{S}(\Omega)$ with $|S| = |\Omega|-1$, there is at most one $x \in \omega$ with $n_Z(S \cup \{x\}) \neq n_Z(S)$, where $\omega \in \Omega$ is the unique skew class such that $\omega \cap S = \emptyset$, and

\item \label{item:char_mm_minor} every minor of $Z$ of order one has at most one circuit.
\end{enumerate}
\end{theorem}
\begin{Proof}
The equivalence of Conditions~\ref{item:char_mm_def} and \ref{item:char_mm_coloop} is shown in \cite{DBLP:journals/tcs/Brijder15}, and the equivalence of Conditions~\ref{item:char_mm_def} and \ref{item:char_mm_nullity} is shown in \cite[Proposition~5.4]{DBLP:journals/siamdm/Bouchet97}.

We now show the equivalence of Condition~\ref{item:char_mm_nullity} and Condition~\ref{item:char_mm_nullity_restr}.

Condition~\ref{item:char_mm_nullity} directly implies Condition~\ref{item:char_mm_nullity_restr}. Conversely, assume Condition~\ref{item:char_mm_nullity_restr} holds. Assume to the contrary that there is an $S \in \mathcal{S}(\Omega)$ and $\omega \in \Omega$ with $\omega \cap S = \emptyset$ such that $n_Z(S \cup \{x_1\}) = n_Z(S \cup \{x_2\}) = n_Z(S) + 1$ for distinct $x_1, x_2 \in \omega$. Let $S' \in \mathcal{S}(\Omega)$ with $S \subseteq S'$, $|S'| = |\Omega|-1$, and $\omega \cap S' = \emptyset$. By the matroid submodularity property of the nullity function applied to the sets $S'$ and $S \cup \{x_i\}$ with $i \in \{1,2\}$, we have $n_Z(S' \cup \{x_i\}) + n_Z(S) \geq n_Z(S') + n_Z(S \cup \{x_i\}) = n_Z(S') + n_Z(S) + 1$. Thus $n_Z(S' \cup \{x_i\}) = n_Z(S') + 1$ for $i \in \{1,2\}$, which contradicts Condition~\ref{item:char_mm_nullity_restr}.


Finally we show the equivalence of Condition~\ref{item:char_mm_coloop} and Condition~\ref{item:char_mm_minor}.

Assume Condition~\ref{item:char_mm_coloop} holds and let $S \in \mathcal{S}(\Omega)$ with $|S| = |\Omega|-1$. Let $\omega \in \Omega$ be the unique skew class disjoint from $S$. Since there is at most one $x \in \omega$ with $\mathcal{C}(Z[S \cup \{x\}]) \neq \mathcal{C}(Z[S])$, $Z|S$ is either of the form $(\omega,\{\omega\},\emptyset)$ or of the form $(\omega,\{\omega\},\{\{x\}\})$ for some $x \in \omega$. Hence Condition~\ref{item:char_mm_minor} holds.

Assume Condition~\ref{item:char_mm_minor} holds. Let $S \in \mathcal{S}(\Omega)$ and $\omega \in \Omega$ with $\omega \cap S = \emptyset$. Assume to the contrary that there are distinct $x,y \in \omega$ with $\mathcal{C}(Z[S \cup \{x\}]) \neq \mathcal{C}(Z[S]) \neq \mathcal{C}(Z[S \cup \{y\}])$. Let $S' \in \mathcal{S}(\Omega)$ with $S \subseteq S'$, $|S'| = |\Omega|-1$ and $S' \cap \omega = \emptyset$. Then $\mathcal{C}(Z[S' \cup \{x\}]) \neq \mathcal{C}(Z[S']) \neq \mathcal{C}(Z[S' \cup \{y\}])$. We have that both $\{x\}$ and $\{y\}$ are circuits of $Z|S'$, which is of order one. This contradicts Condition~\ref{item:char_mm_minor}. Consequently, Condition~\ref{item:char_mm_coloop} holds.
\end{Proof}
Note that by Condition~\ref{item:char_mm_minor} of Theorem~\ref{thm:char_mm}, the class of multimatroids is closed under minor operations (see also \cite{DBLP:journals/combinatorics/Bouchet98}). Also note that $\mathcal{C}(Z[S \cup \{x\}]) \neq \mathcal{C}(Z[S])$ in Theorem~\ref{thm:char_mm} is equivalent to saying that $x$ is not a coloop of matroid $Z[S \cup \{x\}]$.


It is shown in \cite{DBLP:journals/siamdm/Bouchet97}, that if multimatroid $Z$ is nondegenerate, then $\mathcal{B}(Z) = \mathcal{I}(Z) \cap \mathcal{T}(\Omega)$, i.e., every basis is a transversal.

\begin{definition} \label{def:tight}
A multimatroid $Z$ is called \emph{tight} if for every $S \in \mathcal{S}(\Omega)$ with $|S| = |\Omega|-1$, there is an $x \in \omega$ such that $\mathcal{C}(Z[S \cup \{x\}]) \neq \mathcal{C}(Z[S])$, where $\omega \in \Omega$ is the unique skew class such that $S \cap \omega = \emptyset$.
\end{definition}
By Condition~\ref{item:char_mm_coloop} of Theorem~\ref{thm:char_mm}, if $Z$ is tight, then there is exactly one such $x$ of Definition~\ref{def:tight}.

\begin{remark}
The definition of tight in this paper is slightly different from the definition in \cite{DBLP:journals/ejc/Bouchet01}: in \cite{DBLP:journals/ejc/Bouchet01} tightness also requires that $Z$ is nondegenerate. We have chosen to use this more general notion since various results concerning tightness also hold when $Z$ is degenerate.
\end{remark}

\begin{proposition}[Theorem~4.2b of \cite{DBLP:journals/ejc/Bouchet01}] \label{prop:char_tight}
Let $Z$ be a multimatroid. Then the following conditions are equivalent:
\begin{enumerate}
\item \label{item:char_tight_def_tight} $Z$ is tight,

\item \label{item:char_tight_nullity} for every $S \in \mathcal{S}(\Omega)$ with $|S| = |\Omega|-1$, there is an $x \in \omega$ with $n_Z(S \cup \{x\}) = n_Z(S)+1$ (and $n_Z(S\cup \{y\}) = n_Z(S)$ for all $y \in \omega \setminus\{x\}$), where $\omega \in \Omega$ is the unique skew class such that $S \cap \omega = \emptyset$,

\item \label{item:char_tight_emptyOne} no minor of $Z$ is of the form $(\omega,\{\omega\},\emptyset)$,

\item \label{item:char_tight_circ} every minor of $Z$ with nonempty ground set has a circuit, and

\item \label{item:char_tight_circOne} every minor of $Z$ of order one has a circuit.
%
\end{enumerate}
\end{proposition}
\begin{Proof}
Condition~\ref{item:char_tight_nullity} directly implies Condition~\ref{item:char_tight_def_tight}. Conversely, let $Z$ be tight. Let $S \in \mathcal{S}(\Omega)$ with $|S| = |\Omega|-1$ and $\omega \in \Omega$ be the unique skew class such that $S \cap \omega = \emptyset$. Since $Z$ is tight, there is a $y \in \omega$ such that $\mathcal{C}(Z[S \cup \{y\}]) \neq \mathcal{C}(Z[S])$. Let $C \in \mathcal{C}(Z[S \cup \{y\}])$ with $y \in C$. Then $C \setminus \{y\}$ is an independent set of $Z$, and so $n_Z(C) = 1$. We have by the submodularity inequality of matroids, $n_Z(S) + n_Z(C) \leq n_Z(S \cup \{y\}) + n_Z(C \setminus \{y\})$. Since $n_Z(C) = 1$, we have $n_Z(S) + 1 = n_Z(S \cup \{y\})$. Thus Condition~\ref{item:char_tight_nullity} holds.

By Condition~\ref{item:char_mm_minor} of Theorem~\ref{thm:char_mm}, every minor of a multimatroid $Z$ of order one is either of the form $(\omega,\{\omega\},\emptyset)$ or of the form $(\omega,\{\omega\},\{\{u\}\})$ for some $u \in \omega$. Theorem~4.2b of \cite{DBLP:journals/ejc/Bouchet01} says that $Z$ is tight if and only if every minor of $Z$ of order one is of the latter form. This establishes the equivalence of Conditions~\ref{item:char_tight_def_tight}, \ref{item:char_tight_emptyOne}, and \ref{item:char_tight_circOne}.

Finally, Condition~\ref{item:char_tight_circ} trivially implies Condition~\ref{item:char_tight_circOne}. Conversely, if $Z$ has no circuits, then any minor of $Z$ also has no circuits. Thus Condition~\ref{item:char_tight_circOne} trivially implies Condition~\ref{item:char_tight_circ}.
\end{Proof}

By Proposition~\ref{prop:char_tight}, tightness is preserved under minors (see also \cite[Proposition~4.1]{DBLP:journals/ejc/Bouchet01}).

An $(\ell,k)$-\emph{carrier}, for nonnegative integer $\ell$ and positive integer $k$, is a carrier $(U,\Omega)$ such that $|\Omega| = \ell$ and for all $\omega \in \Omega$, $|\omega|=k$. A \emph{$k$-matroid} is a multimatroid over a $(\ell,k)$-carrier for some $\ell$. Note that an (ordinary) matroid corresponds to a $1$-matroid. While the notion of minor for 1-matroids corresponds to the notion of contraction for matroids, the usual notion of minor for matroids can be retrieved when considering a matroid as a special type of 2-matroid (see Section~\ref{sec:sh_mm_ortho} below).

\section{Sheltered multimatroids and orthogonality} \label{sec:sh_mm_ortho}

Let $Z = (U,\Omega,\mathcal{C})$ be a semi-multimatroid and $M$ be a matroid over $U$. We say that $Z$ is \emph{sheltered} by $M$ if $Z[T]=M[T]$ for all $T \in \mathcal{T}(\Omega)$. Note that $M$ and $(U,\Omega)$ together uniquely determine $Z$. However, not every semi-multimatroid (or even multimatroid) is sheltered by a matroid \cite{DBLP:journals/siamdm/Bouchet97}. Also note that if $Z$ is sheltered by $M$ and $U' \subseteq U$, then $Z[U']$ is sheltered by $M[U']$. Moreover, if $Z$ is sheltered by $M$ and $X \in \mathcal{S}(\Omega)$, then $Z|X$ is sheltered by $M\slash X \delm Y$, where $Y = \{ u \in U \mid u \in \omega \setminus X, \omega \cap X \neq \emptyset, \omega \in \Omega\}$. We say that a semi-multimatroid $Z$ is \emph{representable} over some field $\mathbb{F}$ if there is a matroid representable over $\mathbb{F}$ that shelters $Z$ \cite{BT/IsotropicMatI}. Note that representability is preserved under restriction and taking minors. Also note that the above definition of representability for 1-matroids corresponds to the usual definition of representability for matroids. We say that a semi-multimatroid $Z$ is \emph{binary} or \emph{quaternary} if $Z$ is representable over $GF(2)$ or $GF(4)$, respectively.

If $M$ is a matroid and $f: E(M) \to X$ is an injective function, then $f(M)$ denotes the matroid obtained from $M$ by renaming each element $e \in E(M)$ to $f(e)$.
\begin{definition} \label{def:free_sum}
Let $\sigma = (M_1,\ldots,M_k)$ be a sequence of matroids with a common ground set $E$. For each $i \in I = \{1,\ldots,k\}$, let $\varphi_i$ be the function which sends every $e \in E$ to $(e,i)$. For each $e\in E$, let $\omega_e = \{(e,i) \mid i \in I\}$. Let $\Omega = \{ \omega_e \mid e \in E \}$ and $U = \bigcup \Omega$. Then the semi-multimatroid over $(U,\Omega)$ that is sheltered by the direct sum of the matroids $\varphi_i(M_i)$, $i \in I$, is called the \emph{free sum} of $\sigma$.
\end{definition}

For matroids $M_1$ and $M_2$ over some common ground set $E$, we say that $M_1$ and $M_2$ are \emph{orthogonal} if for all $C \in \mathcal{C}(M_1)$ and $C' \in \mathcal{C}(M_2)$ we have $|C \cap C'| \neq 1$.

\begin{remark}
We remark that orthogonality is closely related to the notion of a ``strong map''. We have that matroids $M_1$ and $M_2$ are orthogonal if and only if $M_1$ is a strong map of $M_2^*$, denoted by $M_1 \to M_2^*$. By symmetry of orthogonality, $M_1 \to M_2^*$ if and only if $M_2 \to M_1^*$.
\end{remark}

\begin{theorem}
Let $\sigma = (M_1,\ldots,M_k)$ be a sequence of matroids with a common ground set $E$. Then the free sum of $\sigma$ is a multimatroid (in fact, a $k$-matroid) if and only if the matroids of $\sigma$ are mutually orthogonal.
\end{theorem}
\begin{Proof}
The if-direction is shown in \cite[Proposition 4.3]{DBLP:journals/combinatorics/Bouchet98}.

Conversely, let the free sum $\mathcal{Z}_\sigma$ of $\sigma$ be a multimatroid. Assume to the contrary that there are distinct matroid $M_i$ and $M_j$ of $\sigma$ that are not orthogonal. Thus, there is a circuit $C$ of $M_i$ and a circuit $C'$ of $M_j$ such that $|C \cap C'| = 1$. Let $C \cap C' = \{e\}$. Let $S = \varphi_i(C \setminus \{e\}) \cup \varphi_j(C' \setminus \{e\})$, where $\varphi_i$ and $\varphi_j$ are as in Definition~\ref{def:free_sum}. Then $S$ is a subtransversal of $\mathcal{Z}_\sigma$. Let $\omega_e$ be the skew class corresponding to $e$. Then $\varphi_i(C) \in \mathcal{C}(Z[S \cup \{\varphi_i(e)\}]) \neq \mathcal{C}(Z[S])$ and $\varphi_j(C') \in \mathcal{C}(Z[S \cup \{\varphi_j(e)\}]) \neq \mathcal{C}(Z[S])$. Since $\varphi_i(e), \varphi_j(e) \in \omega_e$ are distinct, we have a contradiction of Condition~\ref{item:char_mm_coloop} of Theorem~\ref{thm:char_mm}.
\end{Proof}

A \emph{transversal $k$-tuple} $\tau$ of $\Omega$ is a sequence $(T_1, \ldots, T_k)$ of $k$ mutually disjoint transversals of $\Omega$.

We recall from, e.g., \cite[Proposition~2.1.11]{Oxley/MatroidBook-2nd} that a matroid $M$ is orthogonal to its dual $M^*$. The free sum of $(M^*,M)$ is denoted by $\mathcal{Z}_M$. For a matroid $M$, we also define the transversal $2$-tuple $\tau(M) = (T_1,T_2)$ of $\Omega$ where $T_i = \{ (e,i) \mid e \in E(M) \}$ for $i \in \{1,2\}$. Note that $\mathcal{Z}_M[T_1]$ is isomorphic to $M^*$ and $\mathcal{Z}_M[T_2]$ is isomorphic to $M$. Note also that $\mathcal{Z}_{M^*}$ is isomorphic to $\mathcal{Z}_{M}$. It is shown in \cite[Corollary~5.5]{DBLP:journals/ejc/Bouchet01} that $\mathcal{Z}_{M}$ is tight for all matroids $M$. Also, $M$ and $\mathcal{Z}_{M}$ have the same number of bases. Indeed, a transversal $T$ of $\mathcal{Z}_{M}$ is a basis of $\mathcal{Z}_{M}$ if and only if $\varphi^{-1}_2(T \cap T_2)$ is a basis of $M$ (conversely, each $X \subseteq E(M)$ is of the form $\varphi^{-1}_2(T \cap T_2)$ for some unique transversal $T$ of $\mathcal{Z}_{M}$), see \cite[Corollary~4.6]{DBLP:journals/combinatorics/Bouchet98}.

Not every tight 2-matroid is isomorphic to $\mathcal{Z}_M$ for some matroid $M$ (it is easy to come up with examples of tight 2-matroids that are not even isomorphic to a free sum of two orthogonal matroids). Note that if $u \in T_1$ then $\mathcal{Z}_{M}|u = \mathcal{Z}_{M \delm u}$, and if $u \in T_2$, then $\mathcal{Z}_{M}|u = \mathcal{Z}_{M \slash u}$, see also \cite[Corollary~5.3]{DBLP:journals/combinatorics/Bouchet98}. Hence matroid deletion and contraction correspond to minors on 2-matroids of the form $\mathcal{Z}_{M}$. Also, we observe that $M$ is representable over some field $\mathbb{F}$ if and only if $\mathcal{Z}_{M}$ is representable over $\mathbb{F}$. In this way, the notions of minor operations and representability carry over naturally from matroids to tight 2-matroids and, indeed, multimatroids in general. Consequently, multimatroids \emph{generalize} matroids.

The following is shown in \cite[Theorem~13]{BH/InterlacePolyDM/14}.
\begin{proposition} [\cite{BH/InterlacePolyDM/14}] \label{prop:unique_tight_k+1_mm}
Let $Z_1 = (U,\Omega,\mathcal{C}_1)$ and $Z_2 = (U,\Omega,\mathcal{C}_2)$ be tight multimatroids with $|\omega| \geq 3$ for all $\omega \in \Omega$. Let $T \in \mathcal{T}(\Omega)$. If $Z_1 - T = Z_2 - T$, then $Z_1 = Z_2$.
\end{proposition}

If for a nondegenerate multimatroid $Z$, there exists such a tight multimatroid $Z'$ such that $Z'-T=Z$ for some transversal $T$ of $Z'$, then we call $Z$ \emph{tightly extendable} and call $Z'$ a \emph{tight extension} of $Z$.

It is shown in \cite{BH/BicyclePenrose} that all $\mathcal{Z}_{M}$ with $M$ a quaternary matroid are tightly extendable. More specifically, it is shown in \cite{BH/BicyclePenrose} that the class of quaternary matroids is a strict subclass of the class of vf-safe \dmatroids, and vf-safe \dmatroids in turn correspond to tightly extendable 2-matroids by \cite[Theorem~16]{BH/InterlacePolyDM/14}.
\begin{proposition} [\cite{BH/BicyclePenrose}] \label{prop:quat_mat}
Let $M$ be a quaternary matroid. Let $T_i = \{ (e,i) \mid e \in E(M) \}$ for $i \in \{1,2,3\}$. Let $\omega_e = \{(e,i) \mid i \in \{1,2,3\}\}$ for all $e \in E(M)$. Let $\Omega = \{ \omega_e \mid e \in E(M) \}$ and $U = \bigcup \Omega$. Then there is a unique tight $3$-matroid $Z$ over $(U,\Omega)$ such that $Z - T_3 = \mathcal{Z}_{M}$.
\end{proposition}
We denote $Z$ of Proposition~\ref{prop:quat_mat} by $\mathcal{Z}_{M,3}$, where subscript $3$ is to indicate that $Z$ is a $3$-matroid (unlike $\mathcal{Z}_{M}$). The quaternary matroid $\varphi_3^{-1}(\mathcal{Z}_{M,3}[T_3])$, where $\varphi_3$ is as in Definition~\ref{def:free_sum}, is the so-called \emph{bicycle matroid} of $M$, see \cite{BH/BicyclePenrose} for a definition. The nullity of the bicycle matroid is called the \emph{bicycle dimension} of $M$.

We remark that Proposition~\ref{prop:quat_mat} holds more generally for a larger class of quaternary 2-matroids, see \cite{BH/BicyclePenrose} where this larger class is defined in \dmatroid terminology. Delta-matroids, which we do not recall in this paper, are essentially equivalent to 2-matroids, however the (usual) notion of representability of \dmatroids is more restrictive.

\section{Orienting transversals} \label{sec:ort_transv}
%
Let $Z$ be a nondegenerate multimatroid. We define $\ort(Z) = \{ T \in \mathcal{T}(\Omega) \mid Z-T \mbox{ is tight} \}$. The transversals in $\ort(Z)$ are called \emph{orienting} in $Z$. The name ``orienting'' is borrowed from the corresponding notion of ``orienting vector'' for isotropic systems defined in \cite{TutteMartinOrientingVectors/Bouchet91}. While we do not define isotropic systems here, we mention that they are equivalent to binary tight $3$-matroids, see \cite{BT/IsotropicMatI}. The notion of orienting vector is in turn borrowed from the corresponding notion of ``Eulerian orientation'' \cite{Vergnas1988367} in the case where the isotropic system can be obtained from circuit partitions of a 4-regular graph.

Note that if $Z$ is the empty multimatroid (i.e., $U = \Omega = \mathcal{C} = \emptyset$), then $\ort(Z) = \{ \emptyset \}$. As another example, we have that $T_3 \in \ort(\mathcal{Z}_{M,3})$ for all quaternary matroids $M$ (with $T_3$ as in Proposition~\ref{prop:quat_mat}), since $\mathcal{Z}_{M,3} - T_3 = \mathcal{Z}_{M}$ is tight.

We now collect a number of results concerning orienting transversals.

\begin{lemma} \label{lem:eul_char_two}
Let $Z$ be a nondegenerate multimatroid and let $T \in \mathcal{T}(\Omega)$. Then $T \in \ort(Z)$ if and only if for every $S \in \mathcal{S}(\Omega)$ with $|S| = |\Omega|-1$ and $S \cap T = \emptyset$, $Z|S$ has a circuit disjoint from $T$.
\end{lemma}
\begin{Proof}
We have $T \in \ort(Z)$ if and only if $Z-T$ is tight. Let $\Omega'$ be the set of skew classes of $Z-T$. Then, $Z-T$ is tight if and only if for every $S \in \mathcal{S}(\Omega')$ with $|S| = |\Omega|-1$, $(Z-T)|S$ contains a circuit. We observe that: (1) for $S \in \mathcal{S}(\Omega)$, we have $S \in \mathcal{S}(\Omega')$ if and only if $S \cap T = \emptyset$, and (2) $(Z-T)|S = (Z|S)-T$ contains a circuit if and only if $Z|S$ has a circuit disjoint from $T$. This obtains the result.
\end{Proof}


The following result provides a characterization of the notion of orienting in the case where $Z$ is tight.
\begin{theorem} \label{thm:eul_char}
Let $Z$ be a tight nondegenerate multimatroid and let $T \in \mathcal{T}(\Omega)$. Then $T \in \ort(Z)$ if and only if for all circuits $C$ of $Z$, $|C \cap T| \neq 1$.
\end{theorem}
\begin{Proof}
Let $T \in \ort(Z)$. Then $Z-T$ is tight. Assume to the contrary that there is a circuit $C$ of $Z$ with $|C \cap T| = 1$. Let $C \cap T = \{x\}$. Consider now a subtransversal $S$ with $C \setminus \{x\} \subseteq S$, $S \cap T = \emptyset$, and $|S| = |\Omega|-1$. Note that such $S$ exists because $(C \setminus \{x\}) \cap T = \emptyset$. Then $Z|S$ contains the circuit $\{x\}$. Since $Z|S$ has only one skew class, there are no other circuits in $Z|S$. Since $\{x\}\subseteq T$, this contradicts the only-if implication of Lemma~\ref{lem:eul_char_two}.

Conversely, assume that for all circuits $C$ of $Z$, $|C \cap T| \neq 1$.
%
%
Assume to the contrary that $Z-T$ is not tight. Hence there is a subtransversal $S$ of $Z$ with $|S| = |\Omega|-1$ and $S \cap T = \emptyset$ and there is a skew class $\omega \in \Omega$ of $Z$ with $S \cap \omega = \emptyset$ such that $\mathcal{C}(Z[S]) = \mathcal{C}(Z[S \cup \{y\}])$ for all $y \in \omega \setminus T$. Since $Z$ is tight, $\mathcal{C}(Z[S]) \subsetneq \mathcal{C}(Z[S \cup \{x\}])$ for the unique $x \in \omega \cap T$. Let $C \in \mathcal{C}(Z[S \cup \{x\}]) \setminus \mathcal{C}(Z[S])$. Then $|C \cap T| = 1$ --- a contradiction.
\end{Proof}

For a nondegenerate multimatroid $Z = (U,\Omega,\mathcal{C})$ and $T \in \mathcal{T}(\Omega)$, define $\ortE_{T,Z} = \ort(Z) \cap 2^{U\setminus T}$. In other words, $\ortE_{T,Z} = \{ Y \in \ort(Z) \mid Y \cap T = \emptyset \}$. For notational convenience, we drop the subscript $Z$ when the multimatroid $Z$ under consideration is clear, writing $\ortE_T$ instead of $\ortE_{T,Z}$.

We denote symmetric difference by $\sdif$.
\begin{lemma} \label{lem:eul_null}
Let $Z$ be a $3$-matroid. Let $T \in \mathcal{T}(\Omega)$, let $x \in T$, and let $p_1$ and $p_2$ be the skew pairs containing $x$. If $n_Z(T \sdif p_1) = n_Z(T \sdif p_2) = n_Z(T)-1$, then $\ortE_{T \sdif p_1} \cap \ortE_{T \sdif p_2} = \emptyset$ and $\ortE_T = \ortE_{T \sdif p_1} \cup \ortE_{T \sdif p_2}$.
\end{lemma}
\begin{Proof}
Let $\omega$ be the skew class containing $x$. Let $p_1 = \{x,x_1\}$ and $p_2 = \{x,x_2\}$.

Let $Y \in \ortE_{T\sdif p_1}$. Then $Y \cap (T\sdif p_1) = \emptyset$ and $Z-Y$ is tight. We have either $x \in Y$ or $x_2 \in Y$, but not both. Consider the subtransversal $S = T \setminus \omega$. Since $Z-Y$ is tight, $n_Z(S) = n_Z(S \cup \{z\})-1$ for some $z \in \omega \setminus Y$. Since $n_Z(S) = n_Z(T)-1$, we have $z = x$ and so $x \notin Y$ and $x_2 \in Y$. Thus, $Y \cap T = \emptyset$ and $Y \cap (T\sdif p_2) \neq \emptyset$. Therefore, $Y \in \ortE_{T}$ and $Y \notin \ortE_{T \sdif p_2}$.

The case where $Y \in \ortE_{T\sdif p_2}$ is analogous. Hence $\ortE_{T \sdif p_1} \cap \ortE_{T \sdif p_2} = \emptyset$ and $\ortE_T \supseteq \ortE_{T \sdif p_1} \cup \ortE_{T \sdif p_2}$. Let now $Y \in \ortE_T$. Then $Y \cap T = \emptyset$. We have either $Y \cap (T \sdif p_1) = \emptyset$ or $Y \cap (T \sdif p_2) = \emptyset$. Thus, $Y \in \ortE_{T \sdif p_1} \cup \ortE_{T \sdif p_2}$.
\end{Proof}


\begin{lemma} \label{lem:max_one_ort}
Let $Z$ be a 3-matroid and let $B \in \mathcal{T}(\Omega)$ be a basis of $Z$. Then $|\ortE_B| \leq 1$.
\end{lemma}
\begin{Proof}
Assume to the contrary that $|\ortE_B| > 1$. Let $T_1,T_2 \in \ortE_B$ with $T_1 \neq T_2$. Hence $T_1 \cap B = \emptyset$, $T_2 \cap B = \emptyset$, $Z-T_1$ is tight, and $Z-T_2$ is tight.

Let $\omega \in \Omega$ with $\omega \cap T_1 \neq \omega \cap T_2$. Let $b,t_1,t_2 \in \omega$ where $b \in B$, $t_1 \in T_1$, and $t_2 \in T_2$. Note that $S = B \setminus \omega$ is a subtransversal of both $Z-T_1$ and $Z-T_2$. Since $Z-T_1$ is a tight 2-matroid, we have, by Condition~\ref{item:char_tight_nullity} of Proposition~\ref{prop:char_tight}, that $n_Z(S \cup \{b\}) \neq n_Z(S \cup \{t_2\})$. Because $Z-T_2$ is a tight 2-matroid, we have $n_Z(S \cup \{b\}) \neq n_Z(S \cup \{t_1\})$. Again by Condition~\ref{item:char_tight_nullity} of Proposition~\ref{prop:char_tight}, $n_Z(S \cup \{t_1\}) = n_Z(S \cup \{t_2\}) = n_Z(S \cup \{b\}) - 1$. Therefore, $n_Z(B) = n_Z(S \cup \{b\}) > 0$ --- a contradiction because $B$ is a basis. Thus $|\ortE_B| \leq 1$.
\end{Proof}

The following lemma concerning orienting transversals will also be useful.
\begin{lemma} \label{lem:singular_ort}
Let $u \in U$ be a singular element of a nondegenerate multimatroid $Z$. Then $\ort(Z) = \{ T \cup \{v\} \mid T \in \ort(Z-\omega), v \in \omega \setminus \{u\} \}$, where $\omega$ is the skew class of $Z$ that contains $u$. In particular, $|\ort(Z)| = (|\omega|-1) \cdot |\ort(Z-\omega)|$.
\end{lemma}
\begin{Proof}
It is easy to observe that for all elements $v$ in a singular skew class $\omega$ of a multimatroid $Z$, we have $Z|v = Z - \omega$ (see also \cite[Proposition 5.5]{DBLP:journals/ejc/Bouchet01}).

Let $T$ be a transversal of $Z$. We have $(Z-T)|u = Z-T-\omega = (Z-\omega)-T = (Z-\omega)-(T \setminus \omega)$.

First assume that $T \in \ort(Z)$. Then $Z-T$ is tight. Since tightness is closed under minors, $(Z-T)|u=(Z-\omega)-(T \setminus \omega)$ is tight. Hence $T \setminus \omega \in \ort(Z-\omega)$.

Assume now that $T \setminus \omega \in \ort(Z-\omega)$. In other words, $(Z-\omega)-(T \setminus \omega)$ is tight. It suffices to show now that $T \in \ort(Z)$ if and only if $u \notin T$. Let $S$ be a subtransversal of $Z-T$ with $|S| = |\Omega|-1$. Let $\omega'$ be the unique skew class of $\Omega$ with $S \cap \omega' = \emptyset$.

If $\omega \neq \omega'$, then since $(Z-\omega)-(T \setminus \omega)$ is tight, there is a $x \in \omega'$ such that $(S \setminus \omega) \cup \{x\}$ contains a circuit $C$ of $(Z-\omega)-(T \setminus \omega)$ with $x \in C$. Therefore, $C \subseteq (S \setminus \omega) \cup \{x\} \subseteq S \cup \{x\}$ is a circuit $C$ of $Z-T$ with $x \in C$.

If $\omega = \omega'$, then $S \cup \{u\}$ contains a circuit $C$ of $Z$ with $u \in C$ (namely $C = \{u\}$). Now, $C$ is a circuit of $Z-T$ if and only if $u \notin T$. Consequently, $Z-T$ is tight if and only if $u \notin T$.
\end{Proof}

We end this section with two lemmas that hold for tight nondegenerate multimatroids.

\begin{lemma} \label{lem:diff_not_1_dep}
Let $Z$ be a tight nondegenerate multimatroid and let $S \in \mathcal{S}(\Omega)$ be nonempty. If for every $C \in \mathcal{C}(Z)$, $S \cup C$ does not contain exactly one skew pair, then $S$ is a dependent set of $Z$ (i.e., $C' \subseteq S$ for some $C' \in \mathcal{C}(Z)$).
\end{lemma}
\begin{Proof}
Assume to the contrary that $S$ is independent. Let $B$ be a basis of $Z$ with $S \subseteq B$. Let $u \in S$ and consider $S' = B \setminus \{u\}$. Let $\omega \in \Omega$ with $u \in \omega$. Since $Z$ is tight and $B$ a transversal (because $Z$ is nondegenerate), $S' \cup \{x\}$ contains a circuit $C$ for some $x \in \omega$. Since $S \cup C$ does not contain exactly one skew pair, we have $x = u$. However, $S' \cup \{u\} = B$ is a basis, so it cannot contain $C$ --- a contradiction.
\end{Proof}

For $W \subseteq U$, let $\mathrm{Sc}(W) = \{ \omega \in \Omega \mid |W \cap \omega| \geq 2 \}$ be the set of skew classes having a skew pair of $W$ as a subset.
\begin{lemma} \label{lem:two_skew_pairs}
Let $Z$ be a tight nondegenerate multimatroid and let $\omega$ and $\omega'$ be distinct skew classes of $Z$ intersecting some circuit $C$ of $Z$. Then there is a circuit $C'$ of $Z$ such that
$\mathrm{Sc}(C \cup C') = \{\omega,\omega'\}$.
\end{lemma}
\begin{Proof}
Since $C$ is a circuit, $C \setminus \omega$ is independent and so there is a basis $B$ with $C \setminus \omega \subseteq B$. Because $B$ is a transversal (as $Z$ is nondegenerate), $C \cup B$ contains exactly one skew pair $p$ (and we have $p \subseteq \omega$). Since $Z$ is tight, $B \sdif p'$ is not a basis for some skew pair $p' \subseteq \omega'$. Hence $B \sdif p'$ contains a circuit $C'$. We have that $C \cup C'$ contains the skew pair $p' \subseteq \omega'$. Since $C \cup C'$ does not contain exactly one skew pair, $C \cup C'$ must contain another skew pair. The only possible skew pair is $p$.
\end{Proof}

Lemma~\ref{lem:two_skew_pairs} is a generalization of the following well-known result (see, e.g., \cite[Proposition~4.2.6]{Oxley/MatroidBook-2nd}) from matroids (i.e., $2$-matroids of the form $Z = \mathcal{Z}_M$) to arbitrary tight multimatroids. In fact, the proof of Lemma~\ref{lem:two_skew_pairs} is similar as the proof of \cite[Proposition~4.2.6]{Oxley/MatroidBook-2nd}.

\begin{corollary}
Let $M$ be a matroid, let $C \in \mathcal{C}(M)$, and let $x,y \in C$ be distinct. Then there is a $C' \in \mathcal{C}(M^*)$ such that $C \cap C' = \{x,y\}$.
\end{corollary}
\begin{Proof}
Let $Z = \mathcal{Z}_M$. Then $\varphi_2(C) \in \mathcal{C}(Z)$, where $\varphi_2$ is as in Definition~\ref{def:free_sum}. Let $\omega$ and $\omega'$ be the skew classes of $Z$ containing $\varphi_2(x)$ and $\varphi_2(y)$, respectively. By Lemma~\ref{lem:two_skew_pairs}, there is a circuit $C''$ of $Z$ such that $\mathrm{Sc}(\varphi_2(C) \cup C'') = \{\omega,\omega'\}$. Since every circuit of $Z = \mathcal{Z}_M$ is either a circuit contained in $\varphi_1(E(M))$ or a circuit contained in $\varphi_2(E(M))$, we have that $C'' \subseteq \varphi_1(E(M))$. Thus $C'' = \varphi_1(C')$ for some circuit $C'$ of $M^*$. Since $\mathrm{Sc}(\varphi_2(C) \cup C'') = \{\omega,\omega'\}$, we have $C \cap C' = \{x,y\}$.
\end{Proof}

\section{Polynomials for Multimatroids} \label{sec:mm_pols}

We recall from \cite{BenardBouchet/TutteMartin1997,BH/InterlacePolyDM/14} the transition polynomial for multimatroids, whose definition is inspired by \cite{Jaeger1990Transition} in the context of 4-regular graphs. We define it here for semi-multimatroids in general.
\begin{definition}
Let $Z$ be a semi-multimatroid over $(U,\Omega)$. We define the \emph{(weighted) transition polynomial} of $Z$ as
$$
Q(Z;\vec{x}, y) = \sum_{T \in \mathcal{T}(\Omega)} \prodv{\vec{x}}{T} y^{n_Z(T)},
$$
where $\vec{x}: U \to R$ is a function from $U$ to some commutative ring $R$, and $\prodv{\vec{x}}{T} = \prod_{u \in T} \vec{x}(u)$.
\end{definition}

For a function $x: U \to R$ and $U' \subseteq U$, we denote the restriction of $x$ to $U'$ by $x[U']$. For notational convenience, for a multimatroid $Z$ with ground set $U$ and a function $x$ with domain $U' \supseteq U$ we write simply $Q(Z;\vec{x}, y)$ to denote $Q(Z;\vec{x}[U], y)$.

\begin{lemma} \label{lem:split_tpol_iter}
Let $Z$ be a semi-multimatroid, let $\vec{x}, \vec{x'}, \vec{x''}: U \to R$ such that $\vec{x} = \vec{x'} + \vec{x''}$. Then $Q(Z;\vec{x}, y) = \sum_{F \subseteq \Omega} Q(Z;x_F, y)$ where, for every $\omega \in \Omega$, $x_F[\omega]$ is equal to $\vec{x'}[\omega]$ if $\omega \in F$ and equal to $\vec{x''}[\omega]$ if $\omega \notin F$.
\end{lemma}
\begin{Proof}
We have $Q(Z;\vec{x}, y) = \sum_{T \in \mathcal{T}(\Omega)} \prodv{\vec{x}}{T} y^{n_Z(T)} = \sum_{T \in \mathcal{T}(\Omega)} \prod_{u \in T} x(u) y^{n_Z(T)} = \sum_{T \in \mathcal{T}(\Omega)} \allowbreak \prod_{u \in T} (x'(u) + x''(u)) y^{n_Z(T)}$. Expansion of $\prod_{u \in T} (x'(u) + x''(u))$ obtains $\sum_{X \subseteq T} \prodv{\vec{x}'}{X}\prodv{\vec{x}''}{T \setminus X} = \sum_{F \subseteq \Omega} \prodv{\vec{x}_F}{T}$, and so we obtain the desired result.
\end{Proof}

\begin{lemma} \label{lem:split_tpol_transv}
Let $Z$ be a semi-multimatroid and let $T \in \mathcal{T}(\Omega)$. Let $\vec{x}, \vec{x'}: U \to R$ with $x'(u) = 0$ for all $u \in U \setminus T$. Then $Q(Z;\vec{x},y) = \sum_{F \subseteq T} \prodv{\vec{x'}}{F} y^{n_Z(F)} Q(Z|F;\vec{x}-\vec{x'},y)$.
\end{lemma}
\begin{Proof}
Let $\vec{x''} = \vec{x}-\vec{x'}$. By Lemma~\ref{lem:split_tpol_iter} and by the natural one-to-one correspondence between subsets $S$ of $\Omega$ and subsets $\{ u \in T \mid u \in \omega \in S, \mbox{ for some } \omega \in \Omega\}$ of a transversal $T \in \mathcal{T}(\Omega)$, we have $Q(Z;\vec{x}, y) = \sum_{F \subseteq T} Q(Z;x_F, y)$ where $x_F[\omega]$ is equal to $\vec{x'}[\omega]$ if $\omega \cap F \neq \emptyset$ and to $\vec{x''}[\omega]$ otherwise. We have $Q(Z;x_F, y) = \sum_{T \in \mathcal{T}(\Omega)} \prodv{\vec{x}_F}{T} y^{n_Z(T)} = \sum_{T' \in \mathcal{T}(\Omega')} \prodv{\vec{x'}}{F} \prodv{\vec{x''}}{T'} y^{n_Z(T' \cup F)}$ with $\Omega' = \{ \omega \in \Omega \mid \omega \cap F = \emptyset \}$. Now, $n_Z(T' \cup F) = n_{Z|F}(T') + n_Z(F)$ and thus $Q(Z;x_F, y) = \prodv{\vec{x'}}{F} y^{n_Z(F)}Q_1(Z|F;x'',y)$. Hence we obtain the desired result.
\end{Proof}

The polynomial $Q(Z;\vec{x}, y)$ with $x(u) = 1$ for all $u \in U$ is denoted by $Q_1(Z;y) = \sum_{T \in \mathcal{T}(\Omega)} y^{n_Z(T)}$.
\begin{lemma} \label{lem:split_tpol_transv_min_one}
Let $Z$ be a nondegenerate semi-multimatroid and let $T \in \mathcal{T}(\Omega)$. Then $Q_1(Z-T;y) = \sum_{F \subseteq T} (-1)^{|F|} y^{n_Z(F)} Q_1(Z|F;y)$.
\end{lemma}
\begin{Proof}
Let $\vec{x}, \vec{x'}, \vec{x''}: U \to R$ be such that (1) $x(u) = 1$ for all $u \in U\setminus T$ and $x(u) = 0$ otherwise, (2) $x'(u) = 0$ for all $u \in U\setminus T$ and $x'(u) = -1$ otherwise, and (3) $x''(u) = 1$ for all $u \in U$. Note that $\vec{x} = \vec{x'} + \vec{x''}$. By the definition of $x$ and the fact that $Z$ is nondegenerate, we have $Q_1(Z-T;y) = Q(Z;\vec{x}, y)$. By Lemma~\ref{lem:split_tpol_transv}, $Q(Z;\vec{x}, y) = \sum_{F \subseteq T} \prodv{\vec{x'}}{F} y^{n_Z(F)} Q(Z|F;\vec{x''},y) = \sum_{F \subseteq T} (-1)^{|F|}y^{n_Z(F)}Q_1(Z|F;y)$.
\end{Proof}

It is interesting to notice that Lemma~\ref{lem:split_tpol_transv_min_one} holds even though $Z$ is in general not unique given $Z-T$.

\begin{lemma} \label{lem:split_tpol_transv_plus_one}
Let $Z$ be a nondegenerate semi-multimatroid and let $T \in \mathcal{T}(\Omega)$. Then $Q_1(Z;y) = \sum_{F \subseteq T} y^{n_Z(F)} Q_1((Z|F)-(T\setminus F);y)$.
\end{lemma}
\begin{Proof}
Let $\vec{x}, \vec{x'}, \vec{x''}: U \to R$ be such that (1) $x(u) = 1$ for all $u \in U$, (2) $x'(u) = 0$ for all $u \in U\setminus T$ and $x'(u) = 1$ otherwise, and (3) $x''(u) = 1$ for all $u \in U \setminus T$ and $x''(u) = 0$ otherwise. Note that $\vec{x} = \vec{x'} + \vec{x''}$. We have $Q_1(Z;y) = Q(Z;\vec{x}, y)$. By Lemma~\ref{lem:split_tpol_transv}, $Q(Z;\vec{x}, y) = \sum_{F \subseteq T} \prodv{\vec{x'}}{F} y^{n_Z(F)} Q(Z|F;\vec{x''},y) = \sum_{F \subseteq T} y^{n_Z(F)} Q_1((Z|F)-(T\setminus F);y)$.
\end{Proof}

Note that Lemma~\ref{lem:split_tpol_transv_min_one} and Lemma~\ref{lem:split_tpol_transv_plus_one} simplify further when $T \in \mathcal{T}(\Omega)$ is a basis. Indeed, in this case we obtain $Q_1(Z-T;y) = \sum_{F \subseteq T} (-1)^{|F|} Q_1(Z|F;y)$ and $Q_1(Z;y) = \sum_{F \subseteq T} Q_1((Z|F)-(T\setminus F);y)$. We remark that a special case of the latter equality in the case where $Z$ is a tight $3$-matroid is given in Lemma~33 of \cite{BH/InterlacePolyDM/14} (formulated there in terms of \dmatroids).

We recall from \cite{BH/InterlacePolyDM/14} the following result (which also follows from Lemma~\ref{lem:split_tpol_transv_min_one} and the observation that $Q_1(Z;1-k) = 0$ for every tight $k$-matroid $Z$ with nonempty ground set).
\begin{proposition} \label{prop:mmpol_res_1-k}
Let $Z$ be a tight $k$-matroid for some $k > 1$ and let $T \in \mathcal{T}(\Omega)$. Then $Q_1(Z-T;1-k) = (-1)^{|\Omega|}(1-k)^{n_Z(T)}$.
\end{proposition}

An interesting property of the polynomial $Q_1$ is that it is connected to the diagonal evaluation of the Tutte polynomial $T(M;x,y)$ of a matroid $M$.
\begin{proposition} [Corollary 25 of \cite{BH/InterlacePolyDM/14}] \label{prop:tutte_as_Q}
Let $M$ be a matroid. Then $T(M;x,x) = Q_1(\mathcal{Z}_{M};x-1)$.
\end{proposition}

With the Propositions~\ref{prop:quat_mat} and \ref{prop:tutte_as_Q} in place, we may specialize Proposition~\ref{prop:mmpol_res_1-k} to the Tutte polynomial for quaternary matroids.
\begin{corollary} [\cite{DBLP:journals/jct/Vertigan98}]
Let $M$ be a quaternary matroid. Then $T(M;-1,-1) = (-1)^{|E(M)|} (-2)^{d}$, where $d$ is the bicycle dimension of $M$.
\end{corollary}

\section{Orienting transversals and binary tight 3-matroids} \label{sec:ort_bin_t3m}

\subsection{Binary tight 3-matroids}


\begin{definition} \label{def:sum_subtransv}
Let $(U,\Omega)$ be an $(\ell,3)$-carrier for some $\ell$. For all $X,Y \in \mathcal{S}(\Omega)$, we define $X+Y := X \sdif Y \sdif (\bigcup \mathrm{Sc}(X \sdif Y))$.
\end{definition}
Note that for a $(\ell,3)$-carrier and $X,Y \in \mathcal{S}(\Omega)$, we have $X + Y \in \mathcal{S}(\Omega)$.

\begin{lemma} \label{lem:sdif_skew_classes}
Let $(U,\Omega)$ be an $(\ell,3)$-carrier and $S,S',S'' \in \mathcal{S}(\Omega)$. Then $\mathrm{Sc}(S \cup (S'+S'')) = \mathrm{Sc}(S \cup S') \sdif \mathrm{Sc}(S \cup S'')$.
\end{lemma}
\begin{Proof}
First, let $\omega \in \mathrm{Sc}(S \cup (S'+S''))$. Then $\omega \cap S = \{x\}$ and $\omega \cap (S'+S'') = \{y\}$ for some distinct $x, y \in \omega$. Let $\omega = \{x,y,z\}$.

If $\omega \cap S' = \emptyset$, then $\omega \cap S'' = \{y\}$ (and similarly if $\omega \cap S'' = \emptyset$). If $\omega \cap S'$ and $\omega \cap S''$ are both nonempty, then they are equal to $\{x\}$ and $\{z\}$ in some order.

Thus, in all cases $\omega$ is in exactly one of $\mathrm{Sc}(S \cup S')$ and $\mathrm{Sc}(S \cup S'')$. Therefore, $\omega \in \mathrm{Sc}(S \cup S') \sdif \mathrm{Sc}(S \cup S'')$.

Conversely, let $\omega \in \mathrm{Sc}(S \cup S') \sdif \mathrm{Sc}(S \cup S'')$. Then $\omega$ is in exactly one of $\mathrm{Sc}(S \cup S')$ and $\mathrm{Sc}(S \cup S'')$. Thus $\omega \cap (S \cup S')$ and $\omega \cap (S \cup S'')$ are equal to $\{x\}$ and $\{x,y\}$ in some order, where $x \in S$. Assume without loss of generality that $\omega \cap (S \cup S') = \{x\}$ and $\omega \cap (S \cup S'') = \{x,y\}$, and therefore $y \in S''$. If $x \notin S'$, then $\omega \cap (S'+S'') = \{y\}$. If $x \in S'$, then $\omega \cap (S'+S'') = \{z\}$. In both cases we obtain $\omega \in \mathrm{Sc}(S \cup (S'+S''))$.
\end{Proof}

Note that, under the assumptions of Lemma~\ref{lem:sdif_skew_classes}, if $|\mathrm{Sc}(S \cup S')|$ and $|\mathrm{Sc}(S \cup S'')|$ are even, then so is $|\mathrm{Sc}(S \cup (S'+S''))|$.

We recall that for a matroid $M$, the \emph{cycle space} $\mathcal{CS}(M)$ of $M$ is the span of the circuits of $M$ under symmetric difference. Each element of $\mathcal{CS}(M)$ is called a \emph{cycle}. It is well known that every cycle of a binary matroid $M$ is the union of some mutually disjoint circuits of $M$ (by convention, the cycle $\emptyset$ is the union of zero circuits).


We recall the following key property of binary tight $3$-matroids.
\begin{proposition} [\cite{BT/IsotropicMatI}] \label{prop:skew_classes_cycles}
Every binary tight $3$-matroid $Z$ is sheltered by a binary matroid $M$ in which every skew class of $Z$ is a cycle of $M$.
\end{proposition}
In Proposition~\ref{prop:bin_tight_implies_IAS} we will recall a standard matrix representation of the binary matroid $M$ of Proposition~\ref{prop:skew_classes_cycles}.

For a multimatroid $Z$, we define the \emph{cycle space} $\mathcal{CS}(Z)$ of $Z$ as $\{ C \in \mathcal{CS}(Z[T]) \mid T \in \mathcal{T}(\Omega) \}$. Each element of $\mathcal{CS}(Z)$ is called a \emph{cycle}. Note that, among cycles of $Z[T]$ with $T \in \mathcal{T}(\Omega)$, the operations of $+$ and $\sdif$ coincide.

\begin{lemma} \label{lem:sum_cycles}
Let $Z$ be a binary tight $3$-matroid. For every $C, C' \in \mathcal{CS}(Z)$, $C+C' \in \mathcal{CS(Z)}$.
\end{lemma}
\begin{Proof}
Follows directly from Proposition~\ref{prop:skew_classes_cycles} and Definition~\ref{def:sum_subtransv} and the fact that for binary matroids $M$, $\mathcal{CS}(M[X]) = \mathcal{CS}(M) \cap 2^X$ for all $X \subseteq E(M)$.
\end{Proof}

\begin{lemma} \label{lem:even_skew_pairs}
Let $Z$ be a binary tight $3$-matroid. For every $C,C' \in \mathcal{CS}(Z)$, $C \cup C'$ contains an even number of skew pairs.
\end{lemma}
\begin{Proof}
Let $X = \{ |\mathrm{Sc}(C \cup C')| \mid C,C' \in \mathcal{CS}(Z) \}$. Assume to the contrary that $X$ contains at least one odd integer. Let $t$ be the smallest odd integer of $X$. Let $C,C' \in \mathcal{CS}(Z)$ with $t = |\mathrm{Sc}(C \cup C')|$. We may assume $C$ and $C'$ are both circuits, since if $C$ is the union of disjoint nonempty cycles $C_1$ and $C_2$, then $C_1 \cup C'$ or $C_2 \cup C'$ contains an odd number of skew pairs not bigger than $t$ (similarly for $C'$). By the definition of a multimatroid, $C \cup C'$ does not contain precisely one skew pair. So $t \geq 3$. Let $\omega_1, \omega_2, \omega_3$ be distinct skew classes that contain skew pairs of $C \cup C'$. By Lemma~\ref{lem:two_skew_pairs} (recall that $C$ and $C'$ are circuits), there is a circuit $C''$ of $Z$ such that $C \cup C''$ contains exactly two skew pairs where one of them is a subset of $\omega_1$ and the other a subset of $\omega_2$. By Lemma~\ref{lem:sum_cycles}, $C''' := C' + C''$ is a cycle. By Lemma~\ref{lem:sdif_skew_classes}, $C \cup C'''$ contains exactly two skew pairs less than $C \cup C'$ --- a contradiction of the minimality of $t$.
\end{Proof}

Lemma~\ref{lem:even_skew_pairs} generalizes the following well-known property of binary matroids, see, e.g., \cite[Theorem~9.1.2(ii)]{Oxley/MatroidBook-2nd}.
\begin{corollary} \label{cor:bin_mat_even_inters}
Let $M$ be a binary matroid. For all $C \in \mathcal{C}(M)$ and $C' \in \mathcal{C}(M^*)$, $|C \cap C'|$ is even.
\end{corollary}
\begin{Proof}
Let $Z = \mathcal{Z}_{M,3}$. Then $\varphi_2(C), \varphi_1(C') \in \mathcal{C}(Z)$, where $\varphi_i$ is as in Definition~\ref{def:free_sum}. By Lemma~\ref{lem:even_skew_pairs}, $\varphi_2(C) \cup \varphi_1(C')$ contains an even number of skew pairs. Thus, $|C \cap C'|$ is even.
\end{Proof}

\begin{lemma} \label{lem:even_in_CS}
Let $Z$ be a binary tight $3$-matroid and let $S \in \mathcal{S}(\Omega)$. If for every $C \in \mathcal{C}(Z)$, $S \cup C$ contains an even number of skew pairs, then $S \in \mathcal{CS}(Z)$.
\end{lemma}
\begin{Proof}
If $S = \emptyset$, then $S$ is the union of zero circuits. Assume that $S$ is nonempty. Assume that for every $C \in \mathcal{C}(Z)$, $S \cup C$ contains an even number of skew pairs, i.e., $|\mathrm{Sc}(S \cup C)|$ is even. By Lemma~\ref{lem:diff_not_1_dep}, $S$ is a dependent set. Hence, $C' \subseteq S$ for some $C' \in \mathcal{C}(Z)$. Consider $S' = S \sdif C'$. Since $S' = S+C'$, we have by Lemma~\ref{lem:sdif_skew_classes} that for all $C \in \mathcal{C}(Z)$, $\mathrm{Sc}(C \cup S') = \mathrm{Sc}(C \cup S) \sdif \mathrm{Sc}(C \cup C')$. Since $|\mathrm{Sc}(C \cup S)|$ is even and $|\mathrm{Sc}(C \cup C')|$ is even by Lemma~\ref{lem:even_skew_pairs}, we have that $|\mathrm{Sc}(C \cup S')|$ is even. By iteration on $S'$, we find that $S$ is the union of some mutually disjoint circuits of $Z$.
\end{Proof}

By Lemma~\ref{lem:even_skew_pairs} and Lemma~\ref{lem:even_in_CS} we immediately obtain the following.
\begin{theorem} \label{thm:max_isotropic_skew_pairs}
Let $Z$ be a binary tight $3$-matroid and let $S \in \mathcal{S}(\Omega)$. Then $S \in \mathcal{CS}(Z)$ if and only if for every $C \in \mathcal{CS}(Z)$, $S \cup C$ contains an even number of skew pairs.
\end{theorem}

\subsection{Orienting transversals for binary tight 3-matroids}

The following result strengthens the only-if direction of Theorem~\ref{thm:eul_char} for the case where $Z$ is a binary tight $3$-matroid. This only-if direction is a reformulation of \cite[Theorem~6.5]{TutteMartinOrientingVectors/Bouchet91} from the context of isotropic systems. We provide a different proof.
\begin{theorem} \label{thm:char_eul_bin}
Let $Z$ be a binary tight $3$-matroid and let $T \in \mathcal{T}(\Omega)$. Then $T \in \ort(Z)$ if and only if for all cycles $C$ of $Z$, $|C \cap T|$ is even.
\end{theorem}
\begin{Proof}
The if direction follows from Theorem~\ref{thm:eul_char} (since every circuit is a cycle).

We now prove the only-if direction. Let $T \in \ort(Z)$. Assume to the contrary that there are cycles of $Z$ for which its intersection with $T$ is of odd cardinality. Let cycle $C$ of $Z$ be such that $|C \cap T|$ is odd and $|C \cap T|$ is minimal among all cycles with this property.
Let $T' \in \mathcal{S}(\Omega)$ disjoint from $T$ such that $C \setminus T \subseteq T'$. Let $\omega \in \Omega$ be such that $\omega \cap (C \cap T) \neq \emptyset$. Since $Z$ is tight, there is an $x \in \omega$ such that $(T' \setminus \omega) \cup \{x\}$ contains a circuit $C'$ with $x \in C'$. If $x \in T$, then $|C' \cap T|=1$ which contradicts Theorem~\ref{thm:eul_char}. Thus, $x \notin T$ and so $x \notin C$. By Lemma~\ref{lem:even_skew_pairs}, $C \cup C'$ contains $k$ skew pairs where $k$ is even. Moreover, $k$ is nonzero because $\omega$ has one such skew pair, namely $\{x\} \cup (\omega \cap T)$, as a subset. For the cycle $C+C'$ we have that $|(C+C') \cap T| = |C \cap T|-k$ is odd, which contradicts the minimality property of $C$.
\end{Proof}

We are interested in $3$-matroids $Z$ such that $\ortE_T \neq \emptyset$ for all $T \in \mathcal{T}(\Omega)$, or equivalently (by Lemma~\ref{lem:eul_null}) such that $\ortE_B \neq \emptyset$ for all $B \in \mathcal{B}(Z)$. The next lemma shows that such a $3$-matroid $Z$ is tight.
\begin{lemma} \label{lem:implies_tight}
Let $Z$ be a 3-matroid with $\ortE_B \neq \emptyset$ for all $B \in \mathcal{B}(Z)$. Then $Z$ is tight.
\end{lemma}
\begin{Proof}
Let $S \in \mathcal{S}(\Omega)$ with $|S| = |\Omega|-1$ and $\omega \in \Omega$ such that $S \cap \omega = \emptyset$. It suffices to show that there is a circuit $C \subseteq S \cup \omega$ with $C \cap \omega \neq \emptyset$. Let $x \in \omega$. Then $\ortE_{S \cup \{x\}} \neq \emptyset$ (recall that, by Lemma~\ref{lem:eul_null}, $\ortE_B \neq \emptyset$ for all $B \in \mathcal{B}(Z)$ implies that $\ortE_T \neq \emptyset$ for all $T \in \mathcal{T}(\Omega)$). Let $T \in \ortE_{S \cup \{x\}}$. Then $Z-T$ is tight. Since $S$ is a subtransversal of $Z-T$, there is a $y \in \omega \setminus T$ such that there is a circuit $C \subseteq S \cup \{y\}$ of $Z-T$ with $y \in C$. Since $C$ is also a circuit of $Z$, the result follows.
\end{Proof}

In the case where $Z$ is a binary tight $3$-matroid, the inequality of Lemma~\ref{lem:max_one_ort} becomes an equality. We remark that this result can also be easily shown using isotropic matroids (see Subsection~\ref{ssec:conseq_int_pol} for a definition of isotropic matroids). However, here we prove it using combinatorial arguments.

\begin{theorem} \label{thm:bin_one_ort}
Let $Z$ be a binary $3$-matroid. Then $Z$ is tight if and only if for each basis $B$ of $Z$ there is a $T \in \ort(Z)$ disjoint from $B$.
\end{theorem}
\begin{Proof}
The if direction follows from Lemma~\ref{lem:implies_tight}.

Conversely, let $Z$ be tight. We prove by induction that $\ortE_T \neq \emptyset$ for all transversals $T$ of $Z$ (and so in particular for all bases $B$ of $Z$). The result trivially holds when the ground set of $Z$ is empty. Assume by the induction hypothesis that $\ortE_{X,Z|x} \neq \emptyset$ for all $x \in U$ and transversals $X$ of $Z|x$. Let $T \in \mathcal{T}(\Omega)$, $u \in T$, and $T' \in \ortE_{T \setminus \omega,Z|u}$, where $\omega \in \Omega$ is such that $u \in \omega$. Let $\omega = \{u,v,w\}$. We show that $T' \cup \{v\}$ or $T' \cup \{w\}$ is in $\ortE_{T,Z}$. Assume to the contrary that both $T' \cup \{v\} \notin \ort(Z)$ and $T' \cup \{w\} \notin \ort(Z)$.

By Theorem~\ref{thm:eul_char}, there are circuits $C$ and $C'$ of $Z$ such that $|C \cap (T' \cup \{v\})| = |C' \cap (T' \cup \{w\})| = 1$. If $v \notin C$, then $C$ is a circuit of $Z|u$ and $|C \cap T'| = |C \cap (T' \cup \{v\})| = 1$, which contradicts that $T' \in \ort(Z|u)$. Thus $v \in C$. Similarly, we obtain $w \in C'$. By Lemma~\ref{lem:even_skew_pairs}, $C \cup C'$ contains an even number, say $k$ number, of skew pairs. One of those pairs is $\{v,w\}$. Since $C \cap T' = C' \cap T' = \emptyset$, we have that $|(C+C') \cap T| = k$. By Lemma~\ref{lem:sum_cycles}, $C + C' \in \mathcal{CS}(Z)$. Since $u \in C + C'$, $C'' := (C + C') \setminus \{u\}$ is a cycle of $Z|u$ (recall that if $M$ is a matroid and $C$ a cycle of $M$ with $u \in C$, then $C \setminus \{u\}$ is a cycle of $M \slash u$). By Theorem~\ref{thm:char_eul_bin}, $|C'' \cap T'|$ is even. This is a contradiction because $|C'' \cap T'| = |(C+C') \cap T| - 1 = k - 1$, where $k$ is even.
\end{Proof}

We remark that, interestingly, the property $\ortE_B \neq \emptyset$ for all $B \in \mathcal{B}(Z)$ of Theorem~\ref{thm:bin_one_ort} is the only reason why we assume in Section~\ref{sec:pol_evals} that the $3$-matroid $Z$ is binary and tight (although we mention that both Statement~\ref{item:eval_min4} and the second equality of Statement~\ref{item:eval_2_rem_p3} of Theorem~\ref{thm:evals} use Proposition~\ref{prop:mmpol_res_1-k} which also requires that $Z$ is tight).

The following result is a consequence of Lemma~\ref{lem:eul_null}, Lemma~\ref{lem:max_one_ort}, and Theorem~\ref{thm:bin_one_ort}.
The obtained result is also a consequence of \cite[Theorem~6.7]{TutteMartinOrientingVectors/Bouchet91} from the context of isotropic systems. The proof given in this paper is shorter and uses matroid-theoretic arguments.
\begin{corollary} [\cite{TutteMartinOrientingVectors/Bouchet91}] \label{cor:eul_coset}
Let $Z$ be a binary tight $3$-matroid and let $T \in \mathcal{T}(\Omega)$. Then $|\ortE_T| = 2^{n_Z(T)}$.
\end{corollary}
\begin{Proof}
The proof is by induction on $n_Z(T)$. If $n_Z(T) = 0$, then the result follows by Lemmas~\ref{lem:max_one_ort} and \ref{thm:bin_one_ort}.

Assume now that $n_Z(T)>0$. Let $C$ be a circuit of $Z[T]$ and let $x \in C$. Let $p_1$ and $p_2$ be the skew pairs that contain $x$. Then $n_Z(T \sdif p_1) = n_Z(T \sdif p_2) = n_Z(T)-1$. By Lemma~\ref{lem:eul_null}, $\ortE_{T \sdif p_1} \cap \ortE_{T \sdif p_2} = \emptyset$ and $\ortE_T = \ortE_{T \sdif p_1} \cup \ortE_{T \sdif p_2}$. By the induction hypothesis, $|\ortE_{T \sdif p_1}| = |\ortE_{T \sdif p_2}| = 2^{n_Z(T)-1}$. Hence $|\ortE_T| = 2^{n_Z(T)}$.
\end{Proof}

As a consequence of Corollary~\ref{cor:eul_coset}, $\ort(Z)$ uniquely determines the bases of a binary tight $3$-matroid $Z$, and so it uniquely determines $Z$.

For $S \in \mathcal{S}(\Omega)$ and $X \subseteq \mathcal{S}(\Omega)$, we define $S+X = \{S+S' \mid S' \in X\}$.

The next result can be deduced from Corollary~6.6 and Theorem~6.7 of \cite{TutteMartinOrientingVectors/Bouchet91} from the context of isotropic systems. More precisely, for $T \in \ort(Z)$, \cite[Corollary 6.6]{TutteMartinOrientingVectors/Bouchet91} shows that $\ort(Z) \subseteq T + \mathcal{CS}(Z-T)$ and \cite[Theorem~6.7]{TutteMartinOrientingVectors/Bouchet91} shows that $\ort(Z) \supseteq T + \mathcal{CS}(Z-T)$. We give here a different proof.
\begin{theorem} \label{thm:char_ort_set_bin}
Let $Z$ be a binary tight $3$-matroid. For every $T \in \ort(Z)$, we have $\ort(Z) = T + \mathcal{CS}(Z-T)$. In particular, $|\ort(Z)| = |\mathcal{CS}(Z-T)|$.
\end{theorem}
\begin{Proof}
Let $T \in \ort(Z)$ and $C \in \mathcal{CS}(Z-T)$. Let $C' \in \mathcal{CS}(Z)$. By Lemma~\ref{lem:even_skew_pairs}, $C \cup C'$ contains an even number of skew pairs. Thus, $|C' \cap T| \equiv |C' \cap (T+C)| \mod 2$. Since $T \in \ort(Z)$, $|C' \cap T|$ is even by Theorem~\ref{thm:char_eul_bin}. Thus, $|C' \cap (T+C)|$ is even for all $C' \in \mathcal{CS}(Z)$. Hence, again by Theorem~\ref{thm:char_eul_bin}, $T+C \in \ort(Z)$. Consequently, $T + \mathcal{CS}(Z-T) \subseteq \ort(Z)$.

Conversely, let $T' \in \ort(Z)$. We need to show that $T+T' \in \mathcal{CS}(Z-T)$. Since $(T+T') \cap T = \emptyset$, it suffices that show that $T+T' \in \mathcal{CS}(Z)$. Let $C \in \mathcal{CS}(Z)$ and let $p$ be the number of skew pairs in $(T+T') \cup C$. By Theorem~\ref{thm:max_isotropic_skew_pairs}, we need to show that $p$ is even. We have $p = |(T \sdif T') \cap C| = |T \cap C| + |T' \cap C| - 2|(T \cap T') \cap C|$. By Theorem~\ref{thm:char_eul_bin}, $|T \cap C|$ and $|T' \cap C|$ are both even. Hence $p$ is even.
\end{Proof}
Note that left-hand sides of the equalities of Theorem~\ref{thm:char_ort_set_bin} are independent of $T$.

The following result from \cite{TutteMartinOrientingVectors/Bouchet91} follows now straightforwardly from Theorem~\ref{thm:char_ort_set_bin}.
\begin{corollary}[Theorem~6.7 of \cite{TutteMartinOrientingVectors/Bouchet91}] \label{cor:ort_coset}
Let $Z$ be a binary tight $3$-matroid. For all $T \in \mathcal{T}(\Omega)$ and $T' \in \ort(Z)$ with $T \cap T' = \emptyset$, $\ort(Z) \cap 2^{U\setminus T} = T' + \mathcal{CS}(Z[T])$.
\end{corollary}
\begin{Proof}
By Theorem~\ref{thm:char_ort_set_bin}, $\ort(Z) \cap 2^{U\setminus T} = (T' + \mathcal{CS}(Z-T')) \cap 2^{U\setminus T}$. For $C \in \mathcal{CS}(Z-T')$, we have $T'+C \subseteq U\setminus T$ if and only if $C \subseteq T$ (since $C \cap T' = \emptyset$). Thus $(T' + \mathcal{CS}(Z-T')) \cap 2^{U\setminus T} = T' + \mathcal{CS}(Z[T])$.
\end{Proof}
In other words, Corollary~\ref{cor:ort_coset} says that for all $T \in \mathcal{T}(\Omega)$, $\ort(Z) \cap 2^{U\setminus T}$ is a coset of $\mathcal{CS}(Z[T])$ under $+$.

\section{Polynomial evaluations for binary tight 3-matroids} \label{sec:pol_evals}

\subsection{Decomposition relation}


The next result is the multivariate version of \cite[Corollary~6.8]{TutteMartinOrientingVectors/Bouchet91} from the context of isotropic systems (which in turn is a generalization of \cite[Proposition~5.1]{LasVergnas1983397} in the context of Eulerian graphs). The proof is similar to that of \cite[Corollary~6.8]{TutteMartinOrientingVectors/Bouchet91}.
\begin{theorem} \label{thm:tight_y_div2}
Let $Z$ be a binary tight $3$-matroid and let $\vec{x}: U \to R$. Then $Q(Z;\vec{x},y) = \sum_{Y \in \ort(Z)} Q(Z-Y;\vec{x},y/2)$.
\end{theorem}
\begin{Proof}
We have $\sum_{Y \in \ort(Z)} Q(Z-Y;\vec{x},y/2) = \sum_{Y \in \ort(Z)}\allowbreak \sum_{T \in \mathcal{T}(\Omega), T\cap Y = \emptyset}  \allowbreak \prodv{\vec{x}}{T} (y/2)^{n_{Z}(T)} = \sum_{T \in \mathcal{T}(\Omega)} \allowbreak \sum_{Y \in \ort(Z), T\cap Y = \emptyset}  \allowbreak \prodv{\vec{x}}{T} (y/2)^{n_{Z}(T)} =  \allowbreak \sum_{T \in \mathcal{T}(\Omega)} \allowbreak 2^{n_Z(T)} \prodv{\vec{x}}{T} (y/2)^{n_{Z}(T)} = Q(Z;\vec{x},y)$, where we used Corollary~\ref{cor:eul_coset} in the second-to-last equality.
\end{Proof}

Theorem~\ref{thm:tight_y_div2} can be stated equivalently as $Q(Z;\vec{x},y) = \sum_{Y \in \ort(Z)} Q(Z;\vec{x}',y/2)$, where $\vec{x}'$ is obtained from $\vec{x}$ by setting all entries indexed by $Y$ to $0$. In this way, we can apply Theorem~\ref{thm:tight_y_div2} recursively to obtain a formula for, say, $Q(Z;\vec{x},2^l)$ (for $l$ a positive integer) in terms of $Q_1(Z;\vec{x},1)$ (cf.\ the proof of Theorem~\ref{thm:evals} below). This illustrates the power of the multivariate approach.

\subsection{Evaluations}
We now very efficiently obtain a number of evaluations of the polynomial $Q$ as consequences of Theorem~\ref{thm:tight_y_div2}. Special cases of many of these evaluations appear in the literature, as explained below.
\begin{theorem} \label{thm:evals}
Let $Z$ be a binary tight $3$-matroid, let $T \in \mathcal{T}(\Omega)$, and let $\vec{x}: U \to R$. Then the following holds.
\begin{enumerate}
\item $Q(Z;\vec{x},2^l) = \sum_{Y_1 \in \ort(Z)} \cdots \sum_{Y_l \in \ort(Z)} \prod_{\omega \in \Omega} \allowbreak \sum_{v \in \omega \setminus (\cup_i Y_i)} x_v$, for any positive integer $l$, \label{item:eval_pow_2}

\item $Q(Z;\vec{x},2) = \sum_{Y \in \ort(Z)} \prod_{u \in Y} \sum_{v \in \omega_u \setminus \{u\}} x_v$, where $\omega_u \in \Omega$ is such that $u \in \omega_u$, \label{item:eval_x_2}

\item $Q_{1}(Z;2) = |\ort(Z)| \cdot 2^{|\Omega|}$, \label{item:eval_2}

\item $Q_{1}(Z-T;2) = \sum_{Y \in \ort(Z)} 2^{|Y\cap T|}$, \label{item:eval_2_rem_p1}

\item $Q_{1}(Z-T;2) = \sum_{F \subseteq T} (-1)^{|F|} \cdot |\ort(Z|F)| \cdot 2^{|T|-r_Z(F)}$, \label{item:eval_2_rem_p2}

\item $Q_{1}(Z-T;2) = k 2^{n_Z(T)} = k |Q_{1}(Z-T;-2)|$ where $k = \sum_{F \subseteq T} (-1)^{|F|} \cdot |\ort(Z|F)| \cdot 2^{r_{Z|F}(T\setminus F)}$ is odd, \label{item:eval_2_rem_p3}

\item $Q_{1}(Z;4) = \sum_{Y_1, Y_2 \in \ort(Z)} 2^{|Y_1\cap Y_2|}$, and \label{item:eval_4}

\item $Q_{1}(Z;-4) = (-1)^{|\Omega|}\sum_{Y \in \ort(Z)} (-2)^{n_Z(Y)}$ \label{item:eval_min4}.
\end{enumerate}
\end{theorem}
\begin{Proof}
Statement~\ref{item:eval_pow_2}. Note that $Q(Z;\vec{x},1) = \prod_{\omega \in \Omega} \sum_{v \in \omega} x_v$. By iteratively applying Theorem~\ref{thm:tight_y_div2} we obtain Statement~\ref{item:eval_pow_2}.

Statement~\ref{item:eval_x_2} is just a slight reformulation of Statement~\ref{item:eval_pow_2} for the case $l=1$.

Statement~\ref{item:eval_2} follows directly from Statement~\ref{item:eval_x_2}.

Statement~\ref{item:eval_2_rem_p1}. Consider $Q(Z;x',y) = Q_1(Z-T;y)$, where $x'(u) = 0$ if $u \in T$ and $x'(u) = 1$ otherwise. Then the result follows from Statement~\ref{item:eval_x_2}.

Statement~\ref{item:eval_2_rem_p2}. By Lemma~\ref{lem:split_tpol_transv_min_one}, we have $Q_1(Z-T;2) = \allowbreak \sum_{F \subseteq T} \allowbreak (-1)^{|F|}2^{n_Z(F)}Q_1(Z|F;2)$. By Statement~\ref{item:eval_2_rem_p1}, $Q_1(Z|F;2) = |\ort(Z|F)| \cdot 2^{|T|-|F|}$. Hence, $Q_1(Z-T;2) = \sum_{F \subseteq T} (-1)^{|F|} \cdot |\ort(Z|F)| \cdot 2^{|T|-(|F| - n_Z(F))}$.

Statement~\ref{item:eval_2_rem_p3}. The second equality follows from Proposition~\ref{prop:mmpol_res_1-k}. We now prove the first equality. By Statement~\ref{item:eval_2_rem_p2}, $k = \sum_{F \subseteq T} s_F$ with $s_F = (-1)^{|F|} \cdot |\ort(Z|F)| \cdot 2^{|T|-r_Z(F)-n_Z(T)}$. We have $|T|-r_Z(F)-n_Z(T) = r_Z(T) - r_Z(F) = r_{Z|F}(T\setminus F)$. Let $F \subseteq T$. Assume $s_F$ is odd. Then $r_{Z|F}(T\setminus F) = 0$. Hence the elements of $T\setminus F$ are all singular in $Z|F$.
By Lemma~\ref{lem:singular_ort}, $|\ort(Z|F)| = 2^{|T\setminus F|}$. Since $s_F$ is odd, $|\ort(Z|F)|$ is odd and therefore $F = T$. Indeed, $s_T = (-1)^{n_Z(T)}$ is odd and we have $s_F$ is odd if and only if $F = T$. Therefore $k$ is odd.

Statement~\ref{item:eval_4} follows from Statement~\ref{item:eval_pow_2} where $l=2$.

Statement~\ref{item:eval_min4}. By Theorem~\ref{thm:tight_y_div2} and Proposition~\ref{prop:mmpol_res_1-k}, $Q_{1}(Z;-4) = \sum_{Y \in \ort(Z)} Q_1(Z-Y;-2) = \sum_{Y \in \ort(Z)} \allowbreak (-1)^{|\Omega|}(-2)^{n_Z(Y)}$.
\end{Proof}

Statement~\ref{item:eval_x_2} of Theorem~\ref{thm:evals} is a generalization of \cite[Proposition~15]{Jaeger1990Transition} from the context of $4$-regular graphs. Statement~\ref{item:eval_2} corresponds to \cite[Corollary~6.9]{TutteMartinOrientingVectors/Bouchet91} concerning the evaluation of $M(S;4)$, where $M(S;y)$ is the global Tutte-Martin polynomial of isotropic system $S$. Statement~\ref{item:eval_2_rem_p1} is a generalization of \cite[Proposition~19]{Jaeger1990Transition}.

Statement~\ref{item:eval_2_rem_p3} corresponds to \cite[Corollary~4.2]{TutteMartinOrientingVectors/Bouchet91} concerning the evaluation of $m(S,T;3)$, where $m(S,T;y)$ is the restricted Tutte-Martin polynomial of isotropic system $S$. The important difference is that the proof in this paper is shortened by using Theorem~\ref{thm:tight_y_div2}, while in \cite{TutteMartinOrientingVectors/Bouchet91} this result is proved \emph{separately} from the isotropic systems counterpart of Theorem~\ref{thm:tight_y_div2} (which is \cite[Corollary~6.8]{TutteMartinOrientingVectors/Bouchet91}) and stretches over 4 pages with separate analysis of multiple special cases. Statement~\ref{item:eval_min4} is a generalization of a result implicit in the proof of \cite[Proposition~5.4]{LasVergnas1983397} and from the context of 4-regular graphs.

\subsection{Consequences for the Tutte polynomial}

Statement~\ref{item:eval_2_rem_p3} in Theorem~\ref{thm:evals} leads to the following result from \cite{TutteMartinOrientingVectors/Bouchet91}.
\begin{corollary}[\cite{TutteMartinOrientingVectors/Bouchet91}]
Let $M$ be a binary matroid. Then $T(M;3,3) = k 2^{d}$, where $d$ is the bicycle dimension of $M$ and $k$ is an odd integer.
\end{corollary}
\begin{Proof}
By Proposition~\ref{prop:tutte_as_Q}, $T(M;x,x) = Q_1(\mathcal{Z}_{M};x-1)$. We have $Q_1(\mathcal{Z}_{M};x-1) = Q_1(\mathcal{Z}_{M,3}-T_3;x-1)$ with $T_3 = \{ (e,3) \mid e \in E(M) \}$. By Statement~\ref{item:eval_2_rem_p3} in Theorem~\ref{thm:evals}, $Q_1(\mathcal{Z}_{M,3}-T_3;2) = k 2^{n_Z(T_3)}$ with $k$ odd. We recall from Section~\ref{sec:sh_mm_ortho} that $n_Z(T_3)$ is the bicycle dimension of $M$ (note that every binary matroid is quaternary).
\end{Proof}

\subsection{Consequences for the interlace polynomial} \label{ssec:conseq_int_pol}

To explain the consequences for the interlace polynomial, we turn to matrices. We consider $X \times Y$-matrices $A$ over some field $\mathbb{F}$ for finite sets $X$ and $Y$, which are matrices where the rows and columns are not ordered, but instead indexed by $X$ and $Y$, respectively (formally, $A$ is a function from $X \times Y$ to $\mathbb{F}$). We denote the rank and nullity of $A$ by $r(A)$ and $n(A)$, respectively. Also, for $X' \subseteq X$ and $Y' \subseteq Y$, we denote by $A[X',Y']$ the submatrix of $A$ by restricting to the rows of $X'$ and columns of $Y'$.

For depictions of matrices that correspond to multimatroids, we use the following convention. Every depiction of a matrix that represents a matroid $M$ that shelters a multimatroid $Z = (U,\Omega,\mathcal{C})$ is such that the order of the columns respects $\Omega$. More precisely,
every matrix depiction of a $k$-matroid consists of $k$ consecutive blocks of $|\Omega|$ columns, where each block corresponds to a transversal and, for all $i \in \{1,\ldots,k\}$, the columns that are at the $i^{th}$ position of a block form a skew class $\omega \in \Omega$.

We recall the following result, which is originally formulated in \cite{bouchet1987} in terms of \dmatroids, see also \cite{BT/IsotropicMatI}.
\begin{proposition}[\cite{bouchet1987}] \label{prop:tight_iff_zero_diag}
Let $Z$ be a $2$-matroid representable over $\mathbb{F}$, i.e., $Z$ is sheltered by a matroid $M$ representable over $\mathbb{F}$. Let $M$ be represented by a matrix over $\mathbb{F}$ of the following form:
$$
\bordermatrix{
& T_1 & T_2 \cr
& I & A
},
$$
with $A$ skew-symmetric (i.e., $A^T = -A$), $\tau = (T_1,T_2)$ a transversal $2$-tuple of $Z$, and $I$ an identity matrix of suitable dimension. Then $A$ is zero-diagonal if and only if $Z$ is tight.
\end{proposition}

For a binary symmetric $V\times V$ matrix $A$, the matroid $M$ that is represented by the binary matrix
$$
\bordermatrix{
& \varphi_1(V) & \varphi_2(V) & \varphi_3(V) \cr
V & I & A & A+I
}
$$
is called the \emph{isotropic matroid} of $A$ \cite{LT/Isotropic_sys_char} (where $\varphi_i$ is again as in Definition~\ref{def:free_sum}). It follows from \cite[Proposition~41]{LT/Isotropic_sys_char}, that $M$ shelters a binary tight $3$-matroid $Z$ over carrier $(U,\Omega)$ with $\Omega = \{ \{ (v,i) \mid i \in \{1,2,3\} \} \mid v \in V \}$ and $U = \bigcup \Omega$. Let us denote $Z$ by $\mathcal{Z}_A$.
Conversely, every binary tight $3$-matroid is sheltered by a matroid isomorphic to an isotropic matroid, see \cite[Theorem~21]{BT/IsotropicMatI}.

\begin{proposition}[Theorem~21 in \cite{BT/IsotropicMatI}] \label{prop:bin_tight_implies_IAS}
Let $Z$ be a $3$-matroid and $T_1$ be a basis of $Z$. Then $Z$ is binary and tight if and only if $Z$ is sheltered by a matroid that is represented by a binary matrix of the form
$$
\bordermatrix{
& T_1 & T_2 & T_3 \cr
& I   & A   & A+I
},
$$
where $A$ is a zero-diagonal symmetric matrix, $I$ is an identity matrix (of suitable dimension), and $\tau = (T_1,T_2,T_3)$ is a transversal $3$-tuple of $Z$.
\end{proposition}

We remark that the condition that $A$ is zero-diagonal in Proposition~\ref{prop:bin_tight_implies_IAS} can be omitted without weakening the result: by interchanging columns of $T_2$ and $T_3$ corresponding to skew pairs, the diagonal entries of $A$ can be arbitrarily set.

We also remark that, because we fix the basis $T_1$ beforehand in Proposition~\ref{prop:bin_tight_implies_IAS}, the formulation here is slightly stronger than what is provided by \cite[Theorem~21]{BT/IsotropicMatI}. This stronger formulation is however implied by the proof of \cite[Theorem~21]{BT/IsotropicMatI}.

Propositions~\ref{prop:tight_iff_zero_diag} and \ref{prop:bin_tight_implies_IAS} allow for an easy alternative proof of the only-if direction of Theorem~\ref{thm:bin_one_ort}.
\begin{Proof}[Alternative proof of the only-if direction of Theorem~\ref{thm:bin_one_ort}.]
Let $Z$ be tight and let $T_1$ be a basis of $Z$. By Proposition~\ref{prop:bin_tight_implies_IAS}, there is a $T_3 \in \ortE_{T_1}$ such that $Z-T_3$ is sheltered by a matroid that is represented by a binary matrix of the form
$$
\bordermatrix{
& T_1 & T_2 \cr
& I & A
},
$$
where $A$ is zero-diagonal. By Proposition~\ref{prop:tight_iff_zero_diag}, $Z-T_3$ is tight. Hence, $T_3 \in \ort(Z)$. Since $T_3$ is disjoint from $T_1$, we are done.
\end{Proof}

We consider graphs where loops are allowed, but not multiple edges. A \emph{graph} $G$ is a tuple $(V,E)$ with $V$ a finite set of \emph{vertices} and $E \subseteq \{ \{x,y\} \mid x,y \in V \}$ a set of \emph{edges}. We denote the set of vertices and the set of edges by $V(G)$ and $E(G)$, respectively. The \emph{adjacency matrix} $A(G)$ of $G$ is the $V(G) \times V(G)$-matrix over $GF(2)$ where for all $x,y \in V$, the entry indexed by $(x,y)$ is $1$ if and only if $\{x,y\} \in E(G)$. Notice that $A(G)$ is symmetric. For $X \subseteq V(G)$, we denote by $G+X$ the graph obtained from $G$ by toggling the existence of loops for the vertices of $X$. For $X \subseteq V(G)$, the \emph{subgraph of $G$ induced by $X$}, denoted by $G[X]$, is $(X, E(G) \cap 2^X)$. The graph $G[X]$ is called an \emph{induced graph} of $G$. Note that $A(G[X]) = A(G)[X,X]$. For convenience, the isotropic matroid of $A(G)$ is simply called the \emph{isotropic matroid} of $G$. Also, we denote $\mathcal{Z}_G := \mathcal{Z}_{A(G)}$.

%

\newcommand{\Eul}{\mathrm{Eul}}
\newcommand{\vodd}{\mathsf{odd}}
\newcommand{\veven}{\mathsf{even}}

A \emph{simple} graph is a graph without loops. A simple graph is called \emph{Eulerian} if every vertex is adjacent to an even number of other vertices. In particular, the empty graph is Eulerian. For a simple graph $G$, define $\Eul(G) = \{ X \subseteq V(G) \mid G[X] \text{ is Eulerian} \}$. Thus $|\Eul(G)|$ is number of Eulerian induced subgraphs of $G$. Also, for $X \subseteq V(G)$, let $\vodd_G(X)$ ($\veven_G(X)$, resp.)\ be the set of vertices in $V(G) \setminus X$ that are adjacent to an odd (even, resp.)\ number of vertices of $X$.

\begin{theorem} \label{thm:ort_as_eul}
Let $G$ be a simple graph. Then
$$
\mathcal{CS}(\mathcal{Z}_G - \varphi_3(V(G))) = \{ \varphi_2(X) \cup \varphi_1(\vodd_G(X)) \mid X \in \mathrm{Eul}(G) \}
$$
and
$$
\ort(\mathcal{Z}_G) = \{ \varphi_1(X) \cup \varphi_2(\vodd_G(X)) \cup \varphi_3(\veven_G(X)) \mid X \in \mathrm{Eul}(G) \}.
$$
In particular, $|\mathcal{CS}(\mathcal{Z}_G - \varphi_3(V(G)))| = |\ort(\mathcal{Z}_G)| = |\Eul(G)|$.
\end{theorem}
\begin{Proof}
The 2-matroid $Z = \mathcal{Z}_G - \varphi_3(V(G))$ is sheltered by a matroid represented by the binary matrix
$$
D = \quad \bordermatrix{
& \varphi_1(V(G)) & \varphi_2(V(G))\cr
V(G) & I   & A(G)
}.
$$
Let $S$ be a subtransversal of $Z$. We have $S = \varphi_1(X_1) \cup \varphi_2(X_2)$ some disjoint $X_1$ and $X_2$. The matrix obtained from $D$ by restricting to the columns of $S$ is
$$
\bordermatrix{
& \varphi_1(X_1) & \varphi_2(X_2)\cr
X_2 & 0 & A(G[X_2]) \cr
X_1 & I & A(G)[X_1,X_2] \cr
X_3 & 0 & A(G)[X_3,X_2]
},
$$
where $X_3 = V(G) \setminus (X_1 \cup X_2)$.

We have that $S \in \mathcal{CS}(Z)$ if and only if the columns of $\varphi_1(X_1) \cup \varphi_2(X_2)$ sum up to zero if and only if both the columns of $A(G[X_2])$ sum up to zero, i.e., $G[X_2]$ is Eulerian, and that $X_1 = \vodd_G(X_2)$.

By Propositions~\ref{prop:tight_iff_zero_diag} and \ref{prop:bin_tight_implies_IAS} and the fact that $G$ is simple, we observe that $\varphi_3(V(G)) \in \ort(\mathcal{Z}_G)$. By Theorem~\ref{thm:char_ort_set_bin}, $\ort(\mathcal{Z}_G) = \varphi_3(V(G)) + \mathcal{CS}(\mathcal{Z}_G - \varphi_3(V(G)))$. Note that $\varphi_3(V(G))+(\varphi_2(X) \cup \varphi_1(\vodd_G(X))) = \varphi_1(X) \cup \varphi_2(\vodd_G(X)) \cup \varphi_3(\veven_G(X))$.
\end{Proof}

\begin{lemma} \label{lem:split_trans_graph}
Let $G$ be a graph. For all $T \in \mathcal{T}(\Omega)$ we define $X_i = \varphi_i^{-1}(T \cap \varphi_i(V(G)))$ with $i \in \{1,2,3\}$, where $\Omega$ is the set of skew classes of $\mathcal{Z}_G$. We have $n_{\mathcal{Z}_G}(T) = n(A(G+X_3[X_2 \cup X_3]))$.
\end{lemma}
\begin{Proof}
By the definition of an isotropic matroid, we have that $n_{\mathcal{Z}_G}(T)$ is the nullity of the matrix
$$
\bordermatrix{
& \varphi_1(X_1) & \varphi_2(X_2) \cup \varphi_3(X_3)\cr
X_1 & I & A(G)[X_1,X_2 \cup X_3] \cr
X_2 \cup X_3 & 0 & A(G+X_3[X_2 \cup X_3])
}.
$$
The nullity of this matrix is in turn equal to the nullity of $A(G+X_3[X_2 \cup X_3])$.
\end{Proof}

The \emph{interlace polynomial} \cite{Arratia2004199} of a graph $G$ is defined as $q(G;y) = \sum_{X \subseteq V(G)} (y-1)^{n(A(G[X]))}$. The \emph{global interlace polynomial} \cite{Aigner200411} of $G$ is defined as $Q(G;y) = \sum_{X \subseteq V(G)} \sum_{Y \subseteq X} (y-2)^{n(A(G+Y[X]))}$.

The following result has been shown in the context of isotropic systems, see \cite[Theorem~8]{Aigner200411} and \cite[Section~5]{DBLP:journals/dm/Bouchet05}.
\begin{proposition} \label{prop:glob_interlace_mm}
For a graph $G$, we have $Q(G;y) = Q_1(\mathcal{Z}_G;y-2)$ and $q(G;y) = Q_1(\mathcal{Z}_G-\varphi_3(V(G));y-1)$.
\end{proposition}
\begin{Proof}
Let $\mathcal{P}_3(V(G))$ be the set of triples $(X_1,X_2,X_3)$ where $X_1 \cup X_2 \cup X_3 = V(G)$ and the $X_i$'s are mutually disjoint. We have $Q(G;y) = \sum_{X \subseteq V(G)} \sum_{Y \subseteq X} (y-2)^{n(A(G+Y[X]))} = \sum_{(X_1, X_2, X_3) \in \mathcal{P}_3(V(G))} (y-2)^{n(A(G+X_3[X_2 \cup X_3]))}$.

By Lemma~\ref{lem:split_trans_graph}, $Q(G;y) = \sum_{P \in \mathcal{P}_3(V(G))} (y-2)^{n_{\mathcal{Z}_G}(T_P)}$, where $T_P = \varphi_1(X_1) \cup \varphi_2(X_2) \cup \varphi_3(X_3)$ for $P = (X_1, X_2, X_3)$. Hence $\sum_{P \in \mathcal{P}_3(V(G))} (y-2)^{n_{\mathcal{Z}_G}(T_P)} = \sum_{T \in \mathcal{T}(\Omega)} (y-2)^{n_{\mathcal{Z}_G}(T)} = Q_1(\mathcal{Z}_G;y-2)$.

The polynomial $q(G;y)$ is obtained from the definition of $Q(G;y)$ by fixing $Y$ to $\emptyset$, i.e., by restricting the summation to those $(X_1, X_2, X_3) \in \mathcal{P}_3(V(G))$ such that $X_3 = \emptyset$. These triples precisely correspond to the transversals $T \in \mathcal{T}(\Omega)$ with $T \cap \varphi_3(V(G)) = \emptyset$.
\end{Proof}

We now translate evaluations of Theorem~\ref{thm:evals} in terms of graphs and the global interlace polynomial.
\begin{corollary}
Let $G$ be a graph. Then $q(G;3) = k|q(G;-1)|$ for some odd integer.

If $G$ is moreover simple, then
\begin{enumerate}
\item $Q(G;4) = |\Eul(G)| \cdot 2^{|V(G)|}$ \cite[Theorem~5]{Aigner200411},

\item $Q(G;6) = \sum_{X_1, X_2 \in \Eul(G)} 2^{|\mathrm{PN}(X_1,X_2)|}$, where $\mathrm{PN}(X_1,X_2) = \{ v \in V(G) \setminus (X_1 \sdif X_2) \mid |N_G(v) \cap X_1| \equiv |N_G(v) \cap X_2| \mod 2 \}$, and

\item $Q(G;-2) = (-1)^{|V(G)|}\sum_{X \in \Eul(G)} (-2)^{n(A(G+\veven_G(X)[V(G)-X]))}$.
\end{enumerate}
\end{corollary}
\begin{Proof}
By Proposition~\ref{prop:glob_interlace_mm}, $q(G;3) = Q_1(\mathcal{Z}_G-\varphi_3(V(G));2)$. By Statement~\ref{item:eval_2_rem_p3} of Theorem~\ref{thm:evals}, $Q_1(\mathcal{Z}_G-\varphi_3(V(G));2) = k|Q_1(\mathcal{Z}_G-\varphi_3(V(G));-2)|$ for some odd integer, which in turn is equal to $k|q(G;-1)|$ by Proposition~\ref{prop:glob_interlace_mm}.

Assume now that $G$ is simple. This way we can invoke Theorem~\ref{thm:ort_as_eul}.

By Proposition~\ref{prop:glob_interlace_mm}, $Q(G;4) = Q_1(\mathcal{Z}_G;2)$. By Statement~\ref{item:eval_2} of Theorem~\ref{thm:evals}, $Q_1(\mathcal{Z}_G;2) = |\ort(\mathcal{Z}_G)|\cdot 2^{|\Omega|}$. The result for $Q(G;4)$ follows now by Theorem~\ref{thm:ort_as_eul} and by noticing that $|V(G)| = |\Omega|$.

By Proposition~\ref{prop:glob_interlace_mm} and Theorem~\ref{thm:evals} we have $Q(G;6) = Q_1(\mathcal{Z}_G;4) = \sum_{Y_1 \in \ort(\mathcal{Z}_G)} \allowbreak \sum_{Y_2 \in \ort(\mathcal{Z}_G)} \allowbreak 2^{|Y_1\cap Y_2|}$. By Theorem~\ref{thm:ort_as_eul}, $Y_1\cap Y_2 = (\varphi_1(X_1) \cap \varphi_1(X_2)) \cup (\varphi_2(\vodd_G(X_1)) \cap \varphi_2(\vodd_G(X_2))) \cup (\varphi_3(\veven_G(X_1)) \cap \varphi_3(\veven_G(X_2)))$ for some $X_1,X_2 \in \mathrm{Eul}(G)$.

Since the $\varphi_i$'s are injective, we have $\varphi_1(X_1) \cap \varphi_1(X_2) = \varphi_1(X_1 \cap X_2)$, $\varphi_2(\vodd_G(X_1)) \cap \varphi_2(\vodd_G(X_2)) = \varphi_2(\vodd_G(X_1) \cap \vodd_G(X_2))$, and $\varphi_3(\veven_G(X_1)) \cap \varphi_3(\veven_G(X_2)) = \varphi_3(\veven_G(X_1) \cap \veven_G(X_2))$. Since the codomains of the $\varphi_i$'s are mutually disjoint, we have $|Y_1\cap Y_2| = |X_1 \cap X_2| + |\vodd_G(X_1) \cap \vodd_G(X_2)| + |\veven_G(X_1) \cap \veven_G(X_2)|$.

Note that since $X_1,X_2 \in \mathrm{Eul}(G)$, $|N_G(v) \cap X_1| \equiv |N_G(v) \cap X_2| \equiv 0 \mod 2$ for all $v \in X_1 \cap X_2$. For $v \in V(G) \setminus (X_1 \cup X_2)$, $|N_G(v) \cap X_1| \equiv |N_G(v) \cap X_2| \mod 2$ if and only if $v \in (\veven_G(X_1) \cap \veven_G(X_2)) \cup (\vodd_G(X_1) \cap \vodd_G(X_2))$. Note that $\veven_G(X_1) \cap \veven_G(X_2)$ and $\vodd_G(X_1) \cap \vodd_G(X_2)$ are disjoint and both a subset of $V(G) \setminus (X_1 \cup X_2)$. Thus, $|Y_1\cap Y_2| = |\mathrm{PN}(X_1,X_2)|$. Hence $\sum_{Y_1 \in \ort(\mathcal{Z}_G)} \allowbreak \sum_{Y_2 \in \ort(\mathcal{Z}_G)} 2^{|Y_1\cap Y_2|} = \sum_{X_1, X_2 \in \Eul(G)} 2^{|\mathrm{PN}(X_1,X_2)|}$.

By Proposition~\ref{prop:glob_interlace_mm} and Theorem~\ref{thm:evals} we have $Q(G;-2) = Q_1(\mathcal{Z}_G;-4) \allowbreak = \allowbreak (-1)^{|\Omega|}\sum_{Y \in \ort(\mathcal{Z}_G)} \allowbreak (-2)^{n_{\mathcal{Z}_G}(Y)}$. By Lemma~\ref{lem:split_trans_graph} and Theorem~\ref{thm:ort_as_eul}, we have for all $Y \in \ort(\mathcal{Z}_G)$, $n_{\mathcal{Z}_G}(Y) = n(A(G+\veven_G(X)[\vodd_G(X) \cup \veven_G(X)])) = n(A(G+\veven_G(X)[V(G)-X]))$ for some $X \in \Eul(G)$.
\end{Proof}
The equality $q(G;3) = k|q(G;-1)|$ is shown in \cite{Aigner200411} for the case where $G$ is simple, which is a significant loss of generality.

The single-variable \emph{bracket polynomial} of a graph $G$ is defined as $b(G;y) = \sum_{Y \subseteq V(G)} \allowbreak y^{n(A(G+Y))}$, see \cite{Traldi/Bracket1/09}. The following result is given in the paragraph above Theorem~47 in \cite{BH/InterlacePolyDM/14}, and is shown below as a consequence of Lemma~\ref{lem:split_tpol_transv_plus_one}.

\begin{theorem}[\cite{BH/InterlacePolyDM/14}]
Let $G$ be a graph. Then $Q(G;y) = \sum_{X \subseteq V(G)} b(G[X];y-2)$.
\end{theorem}
\begin{Proof}
By Proposition~\ref{prop:glob_interlace_mm}, $Q(G;y) = Q_1(\mathcal{Z}_G;y-2)$. Let $M$ be the isotropic matroid of $G$. By the definition of an isotropic matroid, $B = \varphi_1(V(G))$ is a basis of $M$. Hence, $B$ is a basis of $\mathcal{Z}_G$. By Lemma~\ref{lem:split_tpol_transv_plus_one}, $Q_1(\mathcal{Z}_G;y-2) = \sum_{F \subseteq B} Q_1((\mathcal{Z}_G|F)-(B\setminus F);y-2) = \sum_{X \subseteq V(G)} Q_1((\mathcal{Z}_G|\varphi_1(X))-\varphi_1(V(G) \setminus X);y-2)$. Since $\mathcal{Z}_G$ is sheltered by $M$, $\mathcal{Z}_G|\varphi_1(X)$ is sheltered by the matroid $M' = M \slash \varphi_1(X) \delm (\varphi_2(X) \cup \varphi_3(X))$. Note that $M'$ is represented by
$$
\bordermatrix{
& \varphi_1(V(G) \setminus X) & \varphi_2(V(G) \setminus X) & \varphi_3(V(G) \setminus X)\cr
V(G) \setminus X & I & A(G[V(G) \setminus X]) & A(G+(V(G)\setminus X)[V(G) \setminus X])
}.
$$
Consequently, $Z_{X} = (\mathcal{Z}_G|\varphi_1(X))-\varphi_1(V(G) \setminus X)$ is sheltered by a matroid $M''$ represented by
$$
\bordermatrix{
& \varphi_2(X') & \varphi_3(X')\cr
X' & A(G[X']) & A(G+X'[X'])
},
$$
where $X' = V(G) \setminus X$. We observe that $Q_1(Z_{X};y-2) = \sum_{Y \subseteq X'} \allowbreak (y-2)^{n(A(G+Y[X']))} = \sum_{Y \subseteq V(G[X'])} \allowbreak (y-2)^{n(A(G[X']+Y))} = b(G[X'];y-2)$. Thus $Q_1(\mathcal{Z}_G;y-2) = \sum_{X \subseteq V(G)} b(G[V(G) \setminus X];y-2) = \sum_{X \subseteq V(G)} b(G[X];y-2)$.
\end{Proof}


\section{Excluded-minor characterization of binary tight 3-matroids} \label{sec:excl_minor_bin_t3m}

In this section we provide an excluded-minor characterization of binary tight $3$-matroids. We do this by using an excluded-minor result for binary \dmatroids from \cite{Bouchet_1991_67}. While \dmatroids and $2$-matroids are equivalent, a concern here is that the (usual) definition of representability for \dmatroids is more restrictive than the definition of representability for 2-matroids. The notion of representability corresponding to the more restrictive definition for \dmatroids is called in this paper ``strongly representable''. We say that a $2$-matroid $Z$ is \emph{strongly representable} over the field $\mathbb{F}$ if $Z$ is sheltered by a matroid represented by a matrix $D$ over $\mathbb{F}$ of the form
$$
\bordermatrix{
& T_1 & T_2\cr
& I   & A
},
$$
where $A$ is a skew-symmetric matrix (i.e., $A^T = -A$) and $(T_1,T_2)$ is a transversal $2$-tuple of $Z$. %
%
In case $\mathbb{F}$ is of characteristic $2$, then $A$ is skew-symmetric simply means that $A$ is symmetric. We say that the $2$-matroid $Z$ is \emph{strongly binary} if $Z$ is strongly representable over $GF(2)$. Note that strongly binary 2-matroids $Z$ are tightly extendable, because we can extend matrix $D$ to a matrix $D'$
$$
\bordermatrix{
& T_1 & T_2 & T_3\cr
& I   & A   & A+I
},
$$
which represents a matroid isomorphic to an isotropic matroid.

We are now ready to recall the excluded-minor result of \cite{Bouchet_1991_67}, formulated here in terms of $2$-matroids.

\begin{proposition} [\cite{Bouchet_1991_67}] \label{prop:bouchet_excl_minor}
Let $Z$ be a $2$-matroid. Then $Z$ is strongly binary if and only if no minor of $Z$ is isomorphic to any of the following five $2$-matroids:
\begin{itemize}
\item $S_1 = (U_3,\Omega_3,\mathcal{C}_1)$ with $\mathcal{C}_1 = \{ \{1_a,2_b,3_b\}, \{1_b,2_a,3_b\}, \{1_b,2_b,3_a\} \}$,
\item $S_2 = (U_3,\Omega_3,\mathcal{C}_2)$ with $\mathcal{C}_2 = \{ \{1_a,2_a,3_a\} \}$,
\item $S_3 = (U_3,\Omega_3,\mathcal{C}_3)$ with $\mathcal{C}_3 = \{ \{1_a,2_a,3_a\}, \{1_b,2_b,3_b\} \}$,
\item $S_4 = (U_4,\Omega_4,\mathcal{C}_4)$ with $\mathcal{C}_4 = \{ \{1_a,2_b,3_b,4_b\}, \{1_b,2_a,3_b,4_b\}, \{1_b,2_b,3_a,4_b\},\allowbreak \{1_b,2_b,3_b,4_a\},\allowbreak \{1_a,2_a,3_a\}, \{1_a,2_a,4_a\}, \{1_a,3_a,4_a\}, \{2_a,3_a,4_a\} \}$, and
\item $S_5 = (U_4,\Omega_4,\mathcal{C}_5)$ with $\mathcal{C}_5 = \{ T_x \setminus \{i_x\} \mid x \in \{a,b\}, T_x = \{1_x,2_x,3_x,4_x\}, i \in \{1, \ldots, 4\} \}$,
\end{itemize}
where $\Omega_\ell = \{ \{i_a, i_b\} \mid i \in \{1, \ldots, \ell\} \}$ and $U_\ell = \bigcup \Omega_\ell$ for $\ell \in \{3,4\}$.
\end{proposition}
Note that $S_5$ is isomorphic to $\mathcal{Z}_{U_{2,4}}$, where $U_{2,4}$ is the uniform matroid of rank $2$ with $4$ elements.


\begin{lemma} \label{lem:char_strongly_bin}
Let $Z$ be a tight $3$-matroid. Then the following statements are equivalent.
\begin{enumerate}
\item $Z$ is binary,
\item for all transversals $T$ of $Z$, $Z-T$ is strongly binary, and
\item there is a transversal $T$ of $Z$ such that $Z-T$ is strongly binary.
\end{enumerate}
\end{lemma}
\begin{Proof}
Let $Z$ be a binary tight $3$-matroid and $T$ a transversal of $Z$. Since $Z-T$ is a $2$-matroid, it contains a basis $T_1$ of $Z-T$ that is a transversal. Now, $T_1$ is a basis of $Z$ as well. So, by Proposition~\ref{prop:bin_tight_implies_IAS}, $Z$ is sheltered by a matroid $M$ that is represented by a binary matrix of the form
$$
\bordermatrix{
& T_1 & T_2 & T_3 \cr
& I   & A   & A+I
},
$$
where $A$ is a zero-diagonal symmetric matrix. Then $Z-T$ is sheltered by the matroid $M - T$ represented by
$$
\bordermatrix{
& T_1 & (T_2 \cup T_3) \setminus T \cr
& I   & A'
},
$$
where $A'$ is the symmetric matrix obtained from $A$ by setting its diagonal entries for the columns with indices outside $T_2 \setminus T$ to $1$. Thus $Z-T$ is strongly binary.

Conversely, if $Z-T$ is strongly binary, then $Z-T$ is sheltered by a matroid represented by a binary matrix of the form
$$
\bordermatrix{
& T_1 & T'\cr
& I   & A
},
$$
where $A$ is symmetric. By the paragraph above Proposition~\ref{prop:bin_tight_implies_IAS}, the matroid represented by
$$
\bordermatrix{
& T_1 & T' & T \cr
& I   & A  & A+I
},
$$
shelters a binary tight $3$-matroid $Z'$, where $\tau = (T_1,T',T)$ is a transversal $3$-tuple of $Z$. By Proposition~\ref{prop:unique_tight_k+1_mm}, $Z = Z'$.
\end{Proof}

Let $\mathrm{inv}$ denote the unique non-trivial automorphism of $GF(4)$. We say that a $V \times V$-matrix $A$ over $GF(4)$ is \emph{$\mathrm{inv}$-symmetric}, if $\mathrm{inv}(A^T) = A$ (we apply $\mathrm{inv}$ entry-wise here).

The following result regarding an extension of isotropic matroids is closely related to some results in \cite{BH/BicyclePenrose} in the context of \dmatroids.
\begin{theorem} \label{thm:IAS_quat}
Let $A$ be an $\mathrm{inv}$-symmetric $V \times V$-matrix over $GF(4)$. Then the matroid $M$ represented by the matrix
$$
\bordermatrix{
& \varphi_1(V) & \varphi_2(V) & \varphi_3(V)\cr
V & I & A & A+I
}
$$
shelters a tight $3$-matroid $Z$ over $(U,\Omega)$, where $\Omega = \{ \{(v,1),(v,2),(v,3)\} \mid v \in V \}$ and $U = \bigcup \Omega$.
\end{theorem}
\begin{proof}
By the definition of $\mathrm{inv}$-symmetry, the diagonal entries of $A$ are either $0$ or $1$. By the reasoning of the paragraph below Proposition~\ref{prop:bin_tight_implies_IAS}, we (may) assume without loss of generality that $A$ is zero diagonal.

Let $S$ be a subtransversal of $\Omega$ with $|S| = |\Omega|-1$. Let $\omega \in \Omega$ such that $\omega \cap S = \emptyset$. We have $\omega = \{t_1,t_2,t_3\}$ with $t_i \in \varphi_i(V)$ for $i \in \{1,2,3\}$. By Condition~\ref{item:char_mm_nullity_restr} of Theorem~\ref{thm:char_mm} and by the definition of tightness, it suffices to show that the ranks of two of the matroids $M[S \cup \{t_1\}]$, $M[S \cup \{t_2\}]$, and $M[S \cup \{t_3\}]$ are equal and the remaining matroid has a rank that is one smaller.

The restrictions of the matrix to the columns of $S \cup \{t_1\}$, $S \cup \{t_2\}$, and $S \cup \{t_3\}$ are of the form
\begin{center}
$
\bordermatrix{
~ & S \cap \varphi_1(V) & S \setminus \varphi_1(V) & t_1 \cr
  & I   & * & 0 \cr
  & 0   & B & 0 \cr
  & 0   & \rho & 1
}$,
$
\bordermatrix{
~ & S \cap \varphi_1(V) & S \setminus \varphi_1(V) & t_2 \cr
  & I   & * & * \cr
  & 0   & B & \mathrm{inv}(\rho^T) \cr
  & 0   & \rho & 0
}$, and
$
\bordermatrix{
~ & S \cap \varphi_1(V) & S \setminus \varphi_1(V) & t_3 \cr
  & I   & * & * \cr
  & 0   & B & \mathrm{inv}(\rho^T) \cr
  & 0   & \rho & 1
}$
\end{center}
respectively, for some matrix $B$ and row $\rho$. By adding the column of $t_1$ in the first matrix to columns of $S \setminus \varphi_1(V)$ in an appropriate way, it is sufficient to show that the ranks of two of the matrices
\begin{center}
$
E_1 =
\bordermatrix{
~ & S \setminus \varphi_1(V) & t_1 \cr
  & B & 0 \cr
  & 0 & 1
}$,
$
E_2 =
\bordermatrix{
~ & S \setminus \varphi_1(V) & t_2 \cr
  & B & \mathrm{inv}(\rho^T) \cr
  & \rho & 0
}$, and
$
E_3 =
\bordermatrix{
~ & S \setminus \varphi_1(V) & t_3 \cr
  & B & \mathrm{inv}(\rho^T) \cr
  & \rho & 1
}$
\end{center}
are equal and the remaining matrix has a rank that is one smaller. This proof now follows a similar line of reasoning as \cite[Lemma~2]{DBLP:journals/ejc/BalisterBCP02}. Notice that $r(E_1) = r(B)+1$.

We distinguish two cases.

If $\rho$ is not in the row space of $B$, then, because of the $\mathrm{inv}$-symmetry of $A$, $\mathrm{inv}(\rho^T)$ is not in the column space of $B$. Thus the rank of
$
\begin{pmatrix}
B & \mathrm{inv}(\rho^T) \cr
\rho & *
\end{pmatrix}
$
is one larger than the rank of
$
\begin{pmatrix}
B \cr
\rho
\end{pmatrix}
$, which in turn is equal to $r(B)+1 = r(E_1)$.

If $\rho$ is in the row space of $B$, then because of the $\mathrm{inv}$-symmetry of $A$, $\mathrm{inv}(\rho^T)$ is in the column space of $B$. Let vector $v$ be such that $B v = \mathrm{inv}(\rho^T)$.

By appropriately adding columns of $S \setminus \varphi_1(V)$ to the column of $t_2$ in $E_2$ and to the column of $t_3$ in $E_3$, we obtain
$
E_2' =
\bordermatrix{
~ & S \setminus \varphi_1(V) & t_2 \cr
  & B & 0 \cr
  & \rho & \rho v
}$
and
$
E_3' =
\bordermatrix{
~ & S \setminus \varphi_1(V) & t_2 \cr
  & B & 0 \cr
  & \rho & 1+\rho v
}$, respectively. Doing the same for the rows we obtain
$
E_2'' =
\bordermatrix{
~ & S \setminus \varphi_1(V) & t_2 \cr
  & B & 0 \cr
  & 0 & \rho v
}$
and
$
E_3'' =
\bordermatrix{
~ & S \setminus \varphi_1(V) & t_2 \cr
  & B & 0 \cr
  & 0 & 1+\rho v
}$,
respectively. It suffices to show that $\rho v \in \{0,1\}$.

We notice that for an arbitrary $W \times V$-matrix $A_1$ and an $\mathrm{inv}$-symmetric $V \times V$-matrix $A_2$, $A_3 = \mathrm{inv}(A_1^T) A_2 A_1$ is $\mathrm{inv}$-symmetric. Indeed, $\mathrm{inv}(A_3^T) = \mathrm{inv}(A_1^T A_2^T \mathrm{inv}(A_1)) = A_3$. By choosing $A_1 := v$ and $A_2 := B$ we find that $\mathrm{inv}(v^T) B v = \mathrm{inv}(v^T) \mathrm{inv}(\rho^T) = \mathrm{inv}((\rho v)^T)$ is a one-by-one $\mathrm{inv}$-symmetric matrix. Thus $\rho v = \mathrm{inv}((\rho v)^T) \in \{0,1\}$.
\end{proof}

Similarly as before we call $M$ the \emph{isotropic matroid} of $A$ and we denote the tight $3$-matroid $Z$ of Theorem~\ref{thm:IAS_quat} by $\mathcal{Z}_A$.

Let us denote by $\mathscr{H}_{3,3}$ the tight $3$-matroid $\mathcal{Z}_A$, where
$$
A = \begin{pmatrix}
0 & 1 & a \\
1 & 0 & 1 \\
b & 1 & 0
\end{pmatrix}
$$
is an $\mathrm{inv}$-symmetric matrix over $GF(4) = \{0, 1, a, b\}$. We remark that the isotropic matroid of $A$ is isomorphic to $AG(2,3)$, the rank-$3$ affine geometry over $GF(3)$. We use the subscript ``$3,3$'' to signify that it is a $3$-matroid with $3$ skew classes. Notice that every circuit of $\mathscr{H}_{3,3}$ is a transversal. Also notice that $S_1$ ($S_3$, respectively) is isomorphic to $\mathscr{H}_{3,3}-X$ where $X$ is a circuit (basis, respectively) of $\mathscr{H}_{3,3}$. Consequently, by Lemma~\ref{lem:char_strongly_bin}, $\mathscr{H}_{3,3}$ is not binary.


Since the matroid $U_{2,4}$ is quaternary, we have by Proposition~\ref{prop:quat_mat} that the tight $3$-matroid $\mathcal{Z}_{U_{2,4},3}$ is well defined. We easily verify that $\mathcal{Z}_{U_{2,4},3}$ is isomorphic to $\mathcal{Z}_{A'}$ where
$$
A' = \begin{pmatrix}
0 & 0 & a & b \\
0 & 0 & b & a \\
b & a & 0 & 0 \\
a & b & 0 & 0
\end{pmatrix}.
$$

Let $M_{A}$ and $M_{A'}$ be the isotropic matroids of the matrices $A$ and $A'$ above, respectively. One can straightforwardly verify that for any column/row index $v$ of $A'$, there is an isomorphism between $M_{A'}\slash\{(v,3)\}\delm\{(v,1),(v,2)\}$ and $M_{A}$ respecting the skew classes. Consequently, $\mathcal{Z}_{U_{2,4},3}|(v,3)$ is isomorphic to $\mathscr{H}_{3,3}$ for any element $v$ of $U_{2,4}$.

Next, we show that $S_2$ and $S_4$ are not tightly extendable. We remark that for the case of $S_2$ this was already shown in the context of \dmatroids in \cite{BH/PivotLoopCompl/09}. However, for convenience we provide a direct proof here.

The following lemma follows easily from the definition of a multimatroid.
\begin{lemma} \label{lem:nr_of_bases}
Let $Z$ be a nondegenerate multimatroid, $T \in \mathcal{T}(\Omega)$, and $\omega \in \Omega$. Then $T \cup \omega$ contains either $0$, $|\omega|-1$, or $|\omega|$ bases. If $Z$ moreover is tight, then $T \cup \omega$ contains either $0$ or $|\omega|-1$ bases.
\end{lemma}
\begin{Proof}
Recall that since $Z$ is nondegenerate, all bases $B$ of $Z$ are transversals. Let $S = T \setminus \omega$.

Assume that $n_Z(S) > 0$. Then $n_Z(S \cup \{x\}) > 0$ for all $x \in \omega$, and so there is no basis $B$ of $Z$ with $B \subseteq S \cup \omega = T \cup \omega$.

Assume that $n_Z(S) = 0$. Then for all $x \in \omega$ except possibly one, we have $0 = n_Z(S) = n_Z(S \cup \{x\})$. Hence there are $|\omega|-1$ or $|\omega|$ bases $B$ of $Z$ with $B \subseteq T \cup \omega$. If $Z$ is tight, there is a $y \in \omega$ with $n_Z(S) < n_Z(S \cup \{y\})$, and so in this case $T \cup \omega$ contains precisely $|\omega|-1$ bases.
\end{Proof}

The above lemma leads to the following result.
\begin{theorem} \label{thm:odd_even_bases}
Let $Z$ be a tight $k$-matroid for some odd $k \geq 3$.

Let $X \subseteq U$ and let $Y$ be the union of some skew classes of $\Omega$. Let $b_1$ and $b_2$ be the number of bases of $Z$ that are bases of $Z[X]$ and $Z[X \sdif Y]$, respectively. Then $b_1$ and $b_2$ have equal parity.
\end{theorem}
\begin{Proof}
It suffices to show this result for the case where $Y = \omega \in \Omega$ since the general case follows by iteration.

To show that $b_1$ and $b_2$ have equal parity, we show that $b_1 + b_2$ is even. Recall that since $Z$ is nondegenerate, all bases of $Z$ are transversals. We have $b_1+b_2 = |\{B \in \mathcal{B}(Z) \mid B \setminus \omega \subseteq X \setminus \omega \}|$.

Let $H$ be the set of subtransversals $S$ of $Z$ with $|S| = |\Omega|-1$ and $S \subseteq X \setminus \omega$. Then $b_1+b_2 = \sum_{S \in H} |\{B \in \mathcal{B}(Z) \mid B \setminus \omega = S \}|$.

Since $Z$ is tight and $|\omega|-1$ is even, we have by Lemma~\ref{lem:nr_of_bases} that for each $S \in H$ the number of bases of $Z$ that are a subset of $S \cup \omega$ is even. Thus for each $S \in H$, $|\{B \in \mathcal{B}(Z) \mid B \setminus \omega = S \}|$ is even, and therefore $b_1 + b_2$ is even.
\end{Proof}

Recall that a basis of $Z[X]$ or $Z[X \sdif Y]$ may not be a basis of $Z$. Hence the parities of the number of bases of $Z[X]$ and $Z[X \sdif Y]$ may not necessarily coincide.

Theorem~\ref{thm:odd_even_bases} has a number of interesting special cases. For example, $T \in \mathcal{T}(\Omega)$ is a basis of $Z$ if and only if $Z[T \sdif Y]$ has an odd number of bases if and only if $Z-T$ has an odd number of bases (which is the case where $Y$ is the union of all skew classes). Also, if $Z$ is not the empty multimatroid, then $Z$ has an even number of bases (which is the case where $X = U$ and $Y$ is a nonempty union of skew classes).

Since for any multimatroid $Z$, $Q_1(Z;0)$ is the number of bases of $Z$, we have the following corollary to Theorem~\ref{thm:odd_even_bases}.
\begin{corollary}
Let $Z$ be a tight $k$-matroid for some odd $k \geq 3$. Then $Q_1(Z;0)$ is even when $Z$ is nonempty. Moreover, for any transversal $T$ of $Z$, $Q(Z-T;0)$ is odd if and only if $T$ is a basis of $Z$.
\end{corollary}

We remark that Theorem~\ref{thm:odd_even_bases} is closely related to the definitions of loop complementation and dual pivot for vf-safe \dmatroids from \cite{BH/PivotLoopCompl/09}. Indeed, these operations essentially provide a way to obtain $Z-T'$ for some transversal $T'$ of a tight $3$-matroid $Z$ given $Z-T$ for some transversal $T$ of $Z$, see \cite{BH/InterlacePolyDM/14}.

We now recall a basis-exchange property of multimatroids.
\begin{proposition}[Proposition~5.8 of \cite{DBLP:journals/siamdm/Bouchet97}] \label{prop:basis_exch}
Let $Z$ be a nondegenerate multimatroid. For all $T, T' \in \mathcal{B}(Z)$ and $p \subseteq T \sdif T'$ a skew pair, there is a skew pair $q \subseteq T \sdif T'$ (we allow $q=p$) such that $T'\sdif(p\cup q) \in \mathcal{B}(Z)$.
\end{proposition}

\begin{lemma} \label{lem:no_tight_ext_S2_S4}
The 2-matroids $S_2$ and $S_4$ are not tightly extendable.
\end{lemma}
\begin{Proof}
The proofs for $S_2$ and $S_4$ are so similar that we can treat them together. Let $x \in \{2,4\}$. Assume that $Z = (U,\Omega,\mathcal{C})$ is a tight extension of $S_x$ with $\Omega = \{ \{i_a, i_b, i_c\} \mid i \in \{1, \ldots, \ell\} \}$ and $U = \bigcup \Omega$, where $\ell = 3$ when $x=2$ and $\ell = 4$ when $x=4$. Note that all transversals of $S_2$ are bases except for $\{1_a,2_a,3_a\}$. Also note that the bases of $S_4$ are all transversals of $S_4$ with zero or two elements from $\{1_a,2_a,3_a,4_a\}$. Thus $S_x$ has seven bases, and so by Theorem~\ref{thm:odd_even_bases}, $T = \{1_c,\ldots,\ell_c\}$ is a basis of $Z$. Also, $T' = \{1_b,\ldots,\ell_b\}$ is a basis of $Z$. Let $p \subseteq T \sdif T'$ be a skew pair. Thus $p = \{j_b,j_c\}$ for some $j \in \{1, \ldots, \ell\}$. Note that there are exactly four bases $B$ of $S_x$ with $j_b \in B$. Thus by Theorem~\ref{thm:odd_even_bases} with $X:= \{i_a,i_b \mid i \in \{1, \ldots, \ell\}\} \setminus \{j_a\}$ and $Y:= U \setminus \{j_a,j_b,j_c\}$, we obtain that $T \sdif p$ is not a basis of $Z$. Let $q = \{k_b,k_c\} \subseteq T \sdif T'$ be a skew pair distinct from $p$. Note there are exactly two bases $B$ of $S_x$ with $j_b, k_b \in B$. Thus by Theorem~\ref{thm:odd_even_bases} with $X:= \{i_a,i_b \mid i \in \{1, \ldots, \ell\}\} \setminus \{j_a,k_a\}$ and $Y:= U \setminus \{j_a,j_b,j_c,k_a,k_b,k_c\}$, we have that $T \sdif p \sdif q$ is not a basis of $Z$. This contradicts the basis-exchange property of multimatroids, Proposition~\ref{prop:basis_exch}, for bases $T$ and $T'$.
\end{Proof}

Note that the class of 2-matroids that are tightly extendable is minor-closed. We observe that $S_2$ and $S_4$ are excluded minors for this class (because any proper minor of $S_2$ or $S_4$ is strongly binary and therefore tightly extendable). We remark that three other excluded minors for this class are $\mathcal{Z}_{M}$ where $M$ is one of the following matroids: $U_{2,6}$, $P_6$ or $F_7^-$, see \cite{BH/NullityLoopcGraphsDM/10}.

\begin{theorem} \label{thm:binary_tight_3m_minor_char}
Let $Z$ be a tight $3$-matroid. Then $Z$ is binary if and only if no minor of $Z$ is isomorphic to $\mathscr{H}_{3,3}$.
\end{theorem}
\begin{Proof}
First assume that $Z$ is binary. Since $\mathscr{H}_{3,3}$ is not binary, no minor of $Z$ is isomorphic to $\mathscr{H}_{3,3}$.

Now assume that no minor of $Z$ is isomorphic to $\mathscr{H}_{3,3}$. Let $T \in \mathcal{T}(\Omega)$. By Lemma~\ref{lem:char_strongly_bin} it suffices to show that $Z-T$ is strongly binary. Assume to the contrary that $Z-T$ is not strongly binary. By Proposition~\ref{prop:bouchet_excl_minor}, it has a minor isomorphic to some $S_i$. Let $X$ be a subtransversal of $Z-T$ such that $(Z-T)|X$ is isomorphic to some $S_i$. Note that $(Z-T)|X = (Z|X)-T'$, where $T' \subseteq T$ is a transversal of $Z|X$.

Assume that $(Z|X)-T'$ is isomorphic to $S_1$ or $S_3$. Since $S_1$ ($S_3$, respectively) is isomorphic to $\mathscr{H}_{3,3}-Y$, where $Y$ is a circuit (basis, respectively) of $\mathscr{H}_{3,3}$, we have by Proposition~\ref{prop:unique_tight_k+1_mm} and the fact that $Z|X$ is tight, that $Z|X$ is isomorphic to $\mathscr{H}_{3,3}$ --- a contradiction.

Assume that $(Z|X)-T'$ is isomorphic to $S_5$. Since $S_5$ is isomorphic to $\mathcal{Z}_{U_{2,4}}$, $Z|X$ is isomorphic to $\mathcal{Z}_{U_{2,4},3}$. Recall from the paragraph above Lemma~\ref{lem:nr_of_bases} that $\mathcal{Z}_{U_{2,4},3}$ has a minor isomorphic to $\mathscr{H}_{3,3}$. Hence $Z|X$ has a minor isomorphic to $\mathscr{H}_{3,3}$ --- a contradiction.

Assume that $(Z|X)-T'$ is isomorphic to $S_2$ or $S_4$. Then $Z|X$ is a tight extension of $(Z|X)-T'$ --- a contradiction with Lemma~\ref{lem:no_tight_ext_S2_S4}.

Since this exhausts all possible $S_i$'s, we obtain a contradiction.
\end{Proof}
In other words, Theorem~\ref{thm:binary_tight_3m_minor_char} says that the minor-closed set of tight $3$-matroids that do not have (an isomorphic copy of) $\mathscr{H}_{3,3}$ as a minor is equal to the set of binary tight $3$-matroids.

We now observe that no larger minor-closed class of $3$-matroids $Z$ than the class of binary tight $3$-matroids exists that satisfies the crucial property of Theorem~\ref{thm:bin_one_ort} that $\ortE_B \neq \emptyset$ for all bases $B$ of $Z$. Indeed, by Lemma~\ref{lem:implies_tight} $Z$ must be tight. Moreover, $Z$ must be binary since this property does not hold for the (up to isomorphism) unique excluded minor $\mathscr{H}_{3,3}$. In fact, we have $\ort(\mathscr{H}_{3,3}) = \emptyset$.
Indeed, assume to the contrary that there is a $T \in \ort(\mathscr{H}_{3,3})$. Since $\mathscr{H}_{3,3}-T$ is tight and each circuit is transversal, $\mathscr{H}_{3,3}-T$ has $4$ circuits. However, recall that for any transversal $X$ of $\mathscr{H}_{3,3}$, $\mathscr{H}_{3,3}-X$ is isomorphic to $S_1$ or $S_3$. Thus, $\mathscr{H}_{3,3}-X$ has at most $3$ circuits --- a contradiction.

Also, a straightforward weaker notion of representability of multimatroids is the following: we say that a multimatroid $Z$ is \emph{weakly representable} over some field $\mathbb{F}$ if and only if for all transversals $T$ of $Z$, the matroid $Z[T]$ is representable over $\mathbb{F}$. By definition, a multimatroid representable over $\mathbb{F}$ is weakly representable over $\mathbb{F}$. Let us call a multimatroid \emph{weakly binary} if it is weakly representable over $GF(2)$. Note that the important property given in Lemma~\ref{lem:sum_cycles} does not hold for the weakly-binary tight $3$-matroid $\mathscr{H}_{3,3}$ (indeed, there are $C,C' \in \mathcal{C}(\mathscr{H}_{3,3})$ with $|C+C'|=2$, while all circuits of $\mathscr{H}_{3,3}$ are transversals). Thus the notion of representability of multimatroids (involving shelterings by matroids) used in this paper seems to capture just the right level of generality.

We now extend Theorem~\ref{thm:binary_tight_3m_minor_char} by two additional statements.
\begin{theorem} \label{thm:chars_bintight3m}
Let $Z$ be a tight $3$-matroid. Then the following statements are equivalent.
\begin{enumerate}
\item \label{stat:bin_3m} $Z$ is binary,

\item \label{stat:tot_isotropic} for all circuits $C_1$ and $C_2$ of $Z$, $C_1 \cup C_2$ contains an even number of skew pairs,

\item \label{stat:seymour} for all circuits $C_1$ and $C_2$ of $Z$, $C_1 \cup C_2$ does not contain precisely three skew pairs.

\item \label{stat:no_H33_minor} no minor of $Z$ is isomorphic to $\mathscr{H}_{3,3}$.
\end{enumerate}
\end{theorem}
\begin{Proof}
By Lemma~\ref{lem:sum_cycles}, Statement~\ref{stat:bin_3m} implies Statement~\ref{stat:tot_isotropic}. Statement~\ref{stat:tot_isotropic} implies Statement~\ref{stat:seymour} trivially. Assume now that Statement~\ref{stat:seymour} holds and assume to the contrary that $Z|X$ is isomorphic to $\mathscr{H}_{3,3}$ for some subtransversal $X$ of $Z$. Note that $\mathscr{H}_{3,3}$ has circuits $C_1$ and $C_2$ that are mutually disjoint. Moreover, all circuits of $\mathscr{H}_{3,3}$ are of cardinality $3$. Thus $C_1 \cup C_2$ contains precisely $3$ skew pairs. Hence, $Z|X$ contains circuits $C'_1$ and $C'_2$ such that $C'_1 \cup C'_2$ contains precisely $3$ skew pairs. By the definition of minor, $C'_1 \cup X_1$ and $C'_2 \cup X_2$ are circuits of $Z$ for some $X_1,X_2 \subseteq X$. Since $(C'_1 \cup X_1) \cup (C'_2 \cup X_2)$ contains precisely $3$ skew pairs, we obtain a contradiction. Finally, by Theorem~\ref{thm:binary_tight_3m_minor_char}, Statement~\ref{stat:no_H33_minor} implies Statement~\ref{stat:bin_3m}.
\end{Proof}

The following is shown in \cite[Proposition~17]{BT/IsotropicMatI} (which in turn is closely related to \cite[Property~5.2]{Bouchet_1991_67}).
\begin{proposition}[\cite{BT/IsotropicMatI}] \label{prop:IMI_implies_strongly_bin}
Every binary tight $2$-matroid is strongly binary.
\end{proposition}

This leads to a characterization of binary tight $2$-matroids as well.
\begin{theorem} \label{thm:char_binary_tight_2m}
Let $Z$ be a tight $2$-matroid. Then the following statements are equivalent.
\begin{enumerate}
\item \label{stat:binttwomat} $Z$ is binary,

\item \label{stat:binttwomat_strong} $Z$ is strongly binary,

\item \label{stat:binttwomat_even} for all circuits $C$ and $C'$ of $Z$, $C \cup C'$ contains an even number of skew pairs,

\item \label{stat:binttwomat_not_three} for all circuits $C$ and $C'$ of $Z$, $C \cup C'$ does not contain precisely three skew pairs, and

\item \label{stat:binttwomat_minor} no minor of $Z$ is isomorphic to $S_4$ or $S_5$.
\end{enumerate}
\end{theorem}
\begin{Proof}
Statement~\ref{stat:binttwomat_strong} trivially implies Statement~\ref{stat:binttwomat}. Conversely, if the tight $2$-matroid $Z$ is binary, then $Z$ is strongly binary by Proposition~\ref{prop:IMI_implies_strongly_bin}. Hence Statement~\ref{stat:binttwomat} implies Statement~\ref{stat:binttwomat_strong}.

Assume Statement~\ref{stat:binttwomat_strong} holds, i.e., $Z$ is strongly binary. By the beginning of this section, $Z$ has a tight extension $Z'$ that is binary. Since $Z'$ is a binary tight $3$-matroid, we have by Theorem~\ref{thm:chars_bintight3m} that, for all circuits $C_1$ and $C_2$ of $Z$, $C_1 \cup C_2$ contains an even number of skew pairs. Consequently, this also holds for $Z$. Therefore, Statement~\ref{stat:binttwomat_even} holds.

Statement~\ref{stat:binttwomat_even} trivially implies Statement~\ref{stat:binttwomat_not_three}.

Statement~\ref{stat:binttwomat_not_three} implies Statement~\ref{stat:binttwomat_minor}, since each of $S_4$ and $S_5$ has circuits $C$ and $C'$ such that $C \cup C'$ contains precisely three skew pairs.

Assume Statement~\ref{stat:binttwomat_minor} holds. Observe that 2-matroids $S_1$, $S_2$ and $S_3$ are not tight. Since $Z$ is tight and tightness is preserved under taking minors, no minor of $Z$ is isomorphic to $S_1$, $S_2$ or $S_3$. Since also no minor of $Z$ is isomorphic to $S_4$ or $S_5$, we obtain by Proposition~\ref{prop:bouchet_excl_minor} that $Z$ is strongly binary. Hence, Statement~\ref{stat:binttwomat_strong} holds.
\end{Proof}

Since $\mathcal{Z}_M$ is a tight 2-matroid for every matroid $M$, we obtain as a special case of Theorem~\ref{thm:char_binary_tight_2m} the following well-known characterizations of binary matroids, see \cite{Seymour/CocircThree} and, e.g., \cite[Theorems~9.1.2(ii) and 9.1.3(ii)]{Oxley/MatroidBook-2nd}.
\begin{corollary} \label{cor:chars_bintight3m}
Let $M$ be a matroid. Then the following statements are equivalent.
\begin{enumerate}
\item \label{stat:bin_mat} $M$ is binary,

\item \label{stat:cocirc_even} for every circuit $C$ and cocircuit $C^*$ of $M$, $|C \cap C^*|$ is even, and

\item \label{stat:cocirc_not_three} for every circuit $C$ and cocircuit $C^*$ of $M$, $|C \cap C^*| \neq 3$.

\end{enumerate}
\end{corollary}
\begin{Proof}
Statement~\ref{stat:bin_mat} implies Statement~\ref{stat:cocirc_even} by Corollary~\ref{cor:bin_mat_even_inters}. Statement~\ref{stat:cocirc_even} implies Statement~\ref{stat:cocirc_not_three} trivially. Assume Statement~\ref{stat:cocirc_not_three} holds. To show that $M$ is binary, we show that $\mathcal{Z}_M$ is binary. Let $C_1$ and $C_2$ be circuits of $\mathcal{Z}_M$. By definition of $\mathcal{Z}_M$, each of $C_1$ and $C_2$ is either a subset of $\varphi_1(E(M))$ or a subset of $\varphi_2(E(M))$, where $\varphi_i$ is as in Definition~\ref{def:free_sum}. By Theorem~\ref{thm:char_binary_tight_2m}, it suffices to show that $C_1 \cup C_2$ does not contain precisely three skew pairs. Assume to the contrary that $C_1 \cup C_2$ contains precisely three skew pairs. In particular, $C_1 \cup C_2$ contains at least one skew pair, and so $C_1$ and $C_2$ are not subsets of the same $\varphi_i(E(M))$. Without loss of generality, assume that $C_1 \subseteq \varphi_1(E(M))$ and $C_2 \subseteq \varphi_2(E(M))$. So $C^* = \varphi_1^{-1}(C_1)$ is a cocircuit of $M$ and $C = \varphi_2^{-1}(C_2)$ is a circuit of $M$. Since $C_1 \cup C_2$ contains precisely three skew pairs, we have $|C \cap C^*| = 3$. This contradicts Statement~\ref{stat:cocirc_not_three}.
\end{Proof}

\subsection*{Acknowledgements}
We thank anonymous referees for their helpful comments on an earlier version of this paper.

\begingroup
\setlength{\emergencystretch}{8em}
\printbibliography
\endgroup



\end{document}